\documentclass[a4paper,10pt]{article}
\usepackage{eurosym}
\usepackage{makeidx}
\usepackage{amsfonts}
\usepackage{amsmath,mathtools}
\usepackage{graphicx}
\usepackage{amssymb}
\usepackage{mathrsfs}
\usepackage{graphicx}
\usepackage[all,cmtip]{xy}
\usepackage{synttree}
\usepackage{tikz,pgfplots}
\usepackage{color}
\usepackage{hyperref}
\usetikzlibrary{cd}
\usetikzlibrary{matrix}
\newcommand{\level}{\mathrm{level}}
\newcommand{\depth}{\mathrm{depth}}

\providecommand{\U}[1]{\protect\rule{.1in}{.1in}}
\newtheorem{theorem}{Theorem}[section]

\newtheorem{corollary}[theorem]{Corollary}

\newtheorem{definition}[theorem]{Definition}
\newtheorem{example}[theorem]{Example}
\newtheorem{lemma}[theorem]{Lemma}

\newtheorem{proposition}[theorem]{Proposition}

\newtheorem{remark}[theorem]{Remark}

\newenvironment{proof}[1][Proof]{\noindent \emph{#1.} }{\hfill \ \rule{0.5em}{0.5em}}

\makeatletter\@addtoreset{equation}{section}\makeatother
\makeatletter\@addtoreset{figure}{section}\makeatother
\makeatletter\@addtoreset{table}{section}\makeatother
\begin{document}

\title{Geometry of tree-based tensor formats in tensor Banach spaces
}
\author{Antonio Falc\'{o}\footnote{Corresponding author}$^{\,\,1},$ Wolfgang Hackbusch$^{2}$ and Anthony Nouy$%
^{3}$ \\
$^{1}$ ESI International Chair@CEU-UCH, \\
Departamento de Matem\'aticas, F\'{\i}sica y Ciencias
Tecnol\'ogicas,\\
Universidad Cardenal Herrera-CEU, CEU Universities \\
San Bartolom\'e 55,
46115 Alfara del Patriarca (Valencia), Spain\\
e-mail: \texttt{afalco@uchceu.es}\\
$^{2}$ Max-Planck-Institut \emph{Mathematik in den Naturwissenschaften}\\
Inselstr. 22, D-04103 Leipzig, Germany \\
e-mail: \texttt{{wh@mis.mpg.de} }\\
$^{3}$ Centrale Nantes, Nantes Université, \\
LMJL UMR CNRS 6629\\
1 rue de la No\"e,
44321 Nantes Cedex 3, France.\\
e-mail: \texttt{anthony.nouy@ec-nantes.fr}}
\date{}
\maketitle

\begin{abstract}
In the paper \emph{`On the Dirac-Frenkel Variational Principle on Tensor Banach Spaces'}, we provided a geometrical description of manifolds of tensors in Tucker format with fixed multilinear (or Tucker) rank in tensor Banach spaces, that allowed to extend the Dirac-Frenkel variational principle in the framework of topological tensor spaces. 
The purpose of this note is to extend these results to more general tensor formats. More precisely, we provide a new geometrical description of manifolds of tensors in tree-based (or hierarchical) format, also known as tree tensor networks, which are intersections of manifolds of tensors in Tucker format associated with different partitions of the set of dimensions. The proposed geometrical description of tensors in tree-based format is compatible with the one of manifolds of tensors in Tucker format. 
\end{abstract}


\noindent\emph{2010 AMS Subject Classifications:} 15A69, 46B28, 46A32.

\noindent \emph{Key words:} Tensor spaces, Banach manifolds, Tensor formats, Tree-based tensors, Tree tensor networks.

\newpage

\section{Introduction}

Tensor methods are prominent tools in a wide range of applications involving high-dimensional data or functions.  The exploitation of low-rank structures of tensors is the basis of many approximation or dimension reduction methods, see the surveys \cite{KOL09,Bachmayr2016,nouy:2017_morbook, nouy2017handbook,cichocki2016tensor1,cichocki2017tensor2} and monograph \cite{Hackbusch}. Providing a geometrical description of sets of low-rank tensors has many interests. In particular, it allows to devise robust algorithms for optimization \cite{AMS,Uschmajew2020} or construct reduced order models for dynamical systems \cite{KoLu1}. 

A basic low-rank tensor format is the Tucker format. 
Given a collection of $d$ vector spaces $V_\nu$, $\nu\in D := \{1,\hdots,d\}$,
 and the corresponding algebraic tensor space $\mathbf{V}_D = V_1 \otimes \hdots \otimes V_d,$ the set of tensors $\mathfrak{M}_{\mathfrak{r}}(\mathbf{V}_D)$ of tensors  in Tucker format with rank $\mathfrak{r} = (r_1,\hdots,r_d)$ is the set of tensors $\mathbf{v}$ in $\mathbf{V}_D$ such that  $
 \mathbf{v} \in U_1\otimes \hdots \otimes U_d 
 $
 for some subspaces $U_\nu $ in the Grassmann manifold $\mathbb{G}_{r_\nu}(V_\nu)$ of $r_\nu$-dimensional spaces in $V_\nu.$
A geometrical description of $\mathfrak{M}_{\mathfrak{r}}(\mathbf{V}_D)$ has been introduced in \cite{FHN}, providing this set the structure of a $C^\infty$-Banach manifold. 
Tree-based tensor formats \cite{Falco2018SEMA}, also known as tree tensor networks in physics or data science \cite{Orus2019,stoudenmire2016supervised,grelier:2018,Michel2020},
 are more general low-rank tensor formats, also based on subspaces. They include the hierarchical format \cite{HaKuehn2009} or the tensor-train format \cite{Osedelets1}. 
Sets of tensors in tree-based tensor format are the intersection of a collection of sets of tensors in Tucker format associated with a hierarchy of partitions given by a tree. More precisely, given a tree $T_D$ over $D$ (see Definition~\ref{partition_tree} below for a more precise description), we can define a sequence of partitions $\mathcal{P}_1, \hdots , \mathcal{P}_L$ of $D$, with $L$ the depth of the tree, such that each element  in $\mathcal{P}_k$ is a subset of an element of $ \mathcal{P}_{k-1}$ (see example in Figure \ref{figintro}).
For each  partition ${\mathcal{P}_k}$, a tensor in $\mathbf{V}_D$ can be identified with a tensor in $\mathbf{V}_{\mathcal{P}_k} := \bigotimes_{\alpha\in \mathcal{P}_k} \mathbf{V}_{\alpha}$, that allows to define  manifolds  of tensors in Tucker format $\mathfrak{M}_{\mathfrak{r}_k}(\mathbf{V}_{\mathcal{P}_k})$ with $\mathfrak{r}_k \in \mathbb{N}^{\# \mathcal{P}_k}$.  
The set $\mathcal{FT}_{\mathfrak{r}}(\mathbf{V}_D)$ of tensors in $\mathbf{V}_D$ with tree-based rank $\mathfrak{r} = (r_\alpha)_{\alpha \in T_D} \in \mathbb{N}^{\#T_D}$ is then given by 
$$
\mathcal{FT}_{\mathfrak{r}}(\mathbf{V}_D) = \bigcap_{k=1}^L \mathfrak{M}_{\mathfrak{r}_k}(\mathbf{V}_{\mathcal{P}_k})
$$
where $\mathfrak{r}_k = (r_\alpha)_{\alpha\in \mathcal{P}_k}$.
 \begin{figure}[h]
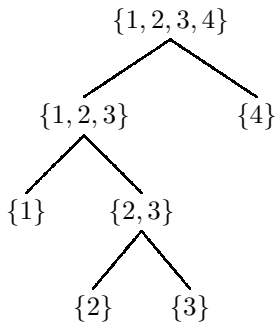

\centering
\synttree[$\{1,2,3,4\}$[$\{1,2,3\}$[$\{1\}$][$\{2,3\}$[$\{2\}$][$\{3\}$]]]
[$\{4\}$]]
\caption{A tree over $D = \{1,2,3,4\}$, with depth $L=3$, and the associated partitions of $D$: $\mathcal{P}_3 = \{\{1\},\{2\},\{3\},\{4\}\}$, $\mathcal{P}_2 = \{\{1\},\{2,3\},\{4\}\}$, $\mathcal{P}_1=\{\{1,2,3\},\{4\}\}$.
}\label{figintro}
\end{figure}

In this paper, we provide a new geometrical description of the sets $\mathcal{FT}_{\mathfrak{r}}(\mathbf{V}_D)$ of tensors with fixed tree-based rank in tensor Banach spaces. This description is compatible with the one of manifolds $ \mathfrak{M}_{\mathfrak{r}_k}(\mathbf{V}_{\mathcal{P}_k})$  
introduced in \cite{FHN}. It is different from the ones from \cite{Uschmajew2012} and \cite{HRS}, respectively introduced for hierarchical and tensor train formats in finite-dimensional tensor spaces. It is also different from the one introduced by the authors in \cite{FHN2015}, that  provided a different chart system. The present geometrical description is more natural and we believe that it is more amenable to understand the geometry and topology of the different tensor formats based on subspaces.
With the present description, and under similar assumptions on the norms of tensor spaces, Theorem 5.2 and Theorem 5.4 in \cite{FHN} also hold for tree-based tensor formats,  that allows to extend the Dirac-Frenkel variational principle for tree-based tensor formats in tensor Banach spaces. 
\\\par 

The outline of this note is as follows. We start in section \ref{sec:prelim} by recalling results from \cite{FHN}. Then in sections\ref{Sec_Hierar} we introduce a description of tree-based tensor formats $\mathcal{FT}_{\mathfrak{r}}(\mathbf{V}_D)$ as an intersection of Tucker formats. Section~\ref{Tucker} is devoted to the geometrical description of manifolds $\mathfrak{M}_{\mathfrak{r}}(\mathbf{V}_D)$ of tensors in Tucker format with fixed rank. Finally in section~\ref{sec:banach_manifold_tucker_fixed_rank}, we introduce the new geometrical description of the manifold $\mathcal{FT}_{\mathfrak{r}}(\mathbf{V}_D)$ of tensors in tree-based tensor format with fixed tree-based. We prove that it is an immersed manifold in the ambient tensor Banach space and outline the extension of the Dirac-Frenkel variational principle for tree-based tensor formats in tensor Banach spaces.

\section{Preliminary results}\label{sec:prelim}

Let $D:=\{1,\ldots,d\}$ be a finite index set and consider  
an algebraic tensor space  $ \mathbf{V}_D = \bigotimes
_{\alpha \in D} V_{\alpha}$ generated from vector spaces $V_{\alpha}$, $\alpha \in D$. 
Concerning the definition of the algebraic tensor space we refer to Greub \cite{Greub}.
For any partition $\mathcal{P}_D$ of $D$, the algebraic tensor space $ \mathbf{V}_D $ can be identified with an algebraic tensor space generated from 
vector spaces $\mathbf{V}_{\alpha}$, $\alpha \in \mathcal{P}_D$. Indeed, for any partition $\mathcal{P}_D$ of $D,$ the equality   
$$
\mathbf{V}_D = \left. \bigotimes_{\alpha \in \mathcal{P}_D} \mathbf{V}_{\alpha} \right.
$$ 
holds, with  $\mathbf{V}_{\alpha} := \left. \bigotimes_{j \in \alpha} V_j \right.$ if $\alpha \neq \{j\},$ for some $j \in D,$ or  
$\mathbf{V}_{\alpha}=V_j$ if $\alpha = \{j\}$ for some $j \in D.$ Next we identify $D$ with the trivial partition 
$\{\{1\},\{2\},\ldots,\{d\}\}.$
\begin{remark}
In \cite{FHN}, we considered  the tensor space $ \mathbf{V}_D = \bigotimes
_{\alpha \in D} V_{\alpha}$ for a given $D$. 
It is not difficult to check that the results from \cite{FHN} remain true when substituting $D$ by any partition $\mathcal{P}_D$, that includes the initial case by identifying  $D$ with the trivial partition 
$\{\{1\},\{2\},\ldots,\{d\}\}.$
More precisely, we can substitute with minor changes along the paper ``$\alpha \in D$'' by ``$\alpha \in \mathcal{P}_D$''.
\end{remark}
Before restating Theorem 3.17 of \cite{FHN} in the present framework, we recall some 
definitions from \cite{FHN}. 

\bigskip

Let $X$ and $Y$ be Banach spaces. We denote by $\mathcal{L}(X,Y)$  
the space of continuous linear mappings from 
$X$ into $Y.$ The corresponding operator norm is written as $\left\Vert
\cdot\right\Vert _{Y\leftarrow X}.$ It is well known
that if $Y$ is a Banach space then $(\mathcal{L}(X,Y),\|\cdot\|_{Y\leftarrow X})$
is also a Banach space.

\bigskip

Let $X$ be
a Banach space. We denote by $\mathbb{G}(X)$ the Grassmann manifold of closed subspaces in $X$ (see Section 2 in \cite{FHN}). More precisely, we say that $U\in \mathbb{G}(X)$ holds if and only if $U$ is a closed subspace in $X$ and there exists a closed subspace $W$ in $X$ such that $X=U\oplus W.$ Every 
finite-dimensional subspace of $X$ belongs to $\mathbb{G}(X),$ and we denote by $\mathbb{G}_{n}(X)$ 
the space of all $n$-dimensional subspaces of $X$ $(n\geq 0).$ From Proposition~2.11 in \cite{FHN}, the Banach space $\mathcal{L}(U,W)$ can be identified with an element of $\mathbb{G}(\mathcal{L}(X,X)).$ Hence it is a closed subspace of $\mathcal{L}(X,X).$

\bigskip

Assume that $\mathcal{P}_D$ is a partition of $D$ and 
$(\mathbf{V}_{\alpha},\|\cdot\|_{\alpha})$ is a normed space 
for each $\alpha \in \mathcal{P}_D.$  Following \cite{FALHACK}, 
it is possible to construct for each $\alpha \in \mathcal{P}_D$ a map
$$
U_{\alpha}^{\min}:\mathbb{V}_D \longrightarrow \mathbb{G}(\mathbf{V}_{\alpha}), \quad \mathbf{v} \mapsto U_{\alpha}^{\min}(\mathbf{v})
$$
which satisfies the following properties:
\begin{enumerate}
\item[i)] $\dim U_{\alpha}^{\min}(\mathbf{v}) < \infty,$ for all $\mathbf{v} \in \mathbb{V}_D.$
\item[ii)] $\mathbf{v} \in \left. \bigotimes_{\alpha \in \mathcal{P}_D} U_{\alpha}^{\min}(\mathbf{v})\right.$ and if
there exist subspaces $\mathbf{U}_{\alpha} \subset \mathbf{V}_{\alpha}$ for each $\alpha\in \mathcal{P}_D$ such that $\mathbf{v} \in  \left. \bigotimes_{\alpha \in \mathcal{P}_D} \mathbf{U}_{\alpha}\right.,$ then
$U_{\alpha}^{\min}(\mathbf{v}) \subset \mathbf{U}_{\alpha}$ for each $\alpha \in \mathcal{P}_D.$ 
\end{enumerate}
The linear subspace $U_{\alpha}^{\min}(\mathbf{v})$ is called a minimal subspace of $\mathbf{v}$ in $\mathbf{V}_D.$ 
In consequence, given a fixed partition $\mathcal{P}_D$ of $D,$ we can define for each $\mathbf{v} \in \mathbf{V}_D$
its $\alpha$-rank as $\dim U_{\alpha}^{\min}(\mathbf{v})$ for $\alpha \in \mathcal{P}_D.$ The $\mathcal{P}_D$-rank for each  $\mathbf{v} \in \mathbf{V}_D$ is given by the tuple $(\dim U_{\alpha}^{\min}(\mathbf{v}))_{\alpha \in \mathcal{P}_D} \in \mathbb{N}^{\#\mathcal{P}_D}.$

\bigskip

Given $\mathfrak{r}=(r_{\alpha})_{\alpha \in \mathcal{P}_D} \in \mathbb{N}^{\#\mathcal{P}_D},$ we define the set of tensors in $\mathbf{V}_{D}$ represented in Tucker format with a fixed rank $\mathfrak{r}$ as
$$
\mathfrak{M}_{\mathfrak{r}}(\mathbf{V}_{D})=\left\{
\mathbf{v} \in \mathbf{V}_{D}: \dim U_{\alpha}^{\min}(\mathbf{v}) = r_{\alpha} \text{ for each } \alpha \in \mathcal{P}_D
\right\}.
$$
A tensor 
$\mathbf{v} \in \mathfrak{M}_{\mathfrak{r}}(%
\mathbf{V}_{D})$ if and only if for each $\alpha \in \mathcal{P}_D$ there exists a unique subspace $U_{\alpha}^{\min}(\mathbf{v})  \in \mathbb{G}_{r_{\alpha}}(\mathbf{V}_{\alpha})$
such that $\mathbf{v} \in \left. \bigotimes_{\alpha \in \mathcal{P}_D}U_{\alpha}^{\min}(\mathbf{v}) \right..$ 
Observe, that
$$
\mathfrak{M}_{\mathfrak{r}}\left(\left. \bigotimes_{\alpha \in \mathcal{P}_D}U_{\alpha}^{\min}(\mathbf{v}) \right.\right)=\left\{
\mathbf{v} \in \left. \bigotimes_{\alpha \in \mathcal{P}_D}U_{\alpha}^{\min}(\mathbf{v}) \right.: \dim U_{\alpha}^{\min}(\mathbf{v}) = r_{\alpha} \text{ for each } \alpha \in \mathcal{P}_D
\right\}
$$
is the set of full rank tensors in the finite dimensional space $\left. \bigotimes_{\alpha \in \mathcal{P}_D}U_{\alpha}^{\min}(\mathbf{v}) \right..$ Clearly, $\left. \bigotimes_{\alpha \in \mathcal{P}_D}U_{\alpha}^{\min}(\mathbf{v}) \right.$ is also a normed space and it can be shown that $\mathfrak{M}_{\mathfrak{r}}\left(\left. \bigotimes_{\alpha \in \mathcal{P}_D}U_{\alpha}^{\min}(\mathbf{v}) \right.\right)$ is an open set in $\left. \bigotimes_{\alpha \in \mathcal{P}_D}U_{\alpha}^{\min}(\mathbf{v})\right.,$ and hence a manifold.

\bigskip

Recall that for each fixed $\alpha \in \mathcal{P}_D,$ the 
finite dimensional vector space $U_{\alpha}^{\min}(\mathbf{v})$ is linearly isomorphic 
to the vector space $$ \mathbb{R}^{\dim U_{\alpha}(\mathbf{v})} = \mathbb{R}^{r_{\alpha}}$$ for all $\mathbf{v} \in \mathfrak{M}_{\mathfrak{r}}(\mathbf{V}_{D}).$
Hence the finite dimensional vector space $\left. \bigotimes_{\alpha \in \mathcal{P}_D}U_{\alpha}^{\min}(\mathbf{v})\right.$ is linearly isomorphic to the vector space $\mathbb{R}^{
\mathop{\mathchoice{\raise-0.22em\hbox{\huge $\times$}}
{\raise-0.05em\hbox{\Large $\times$}}{\hbox{\large
$\times$}}{\times}}_{\alpha \in \mathcal{P}_D}r_{\alpha}}.$ 
This fact allows to identify the open set of full rank tensors 
in $\mathbb{R}^{\mathop{\mathchoice{\raise-0.22em\hbox{\huge $\times$}}
{\raise-0.05em\hbox{\Large $\times$}}{\hbox{\large
$\times$}}{\times}}_{\alpha \in \mathcal{P}_D}r_{\alpha}},$ denoted by $\mathbb{R}_*^{
\mathop{\mathchoice{\raise-0.22em\hbox{\huge $\times$}}
{\raise-0.05em\hbox{\Large $\times$}}{\hbox{\large
$\times$}}{\times}}_{\alpha \in \mathcal{P}_D}r_{\alpha}},$ with $\mathfrak{M}_{\mathfrak{r}}\left(\left. \bigotimes_{\alpha \in \mathcal{P}_D}U_{\alpha}^{\min}(\mathbf{v}) \right.\right).$

\section{The set of tensors in tree-based format with fixed tree-based rank}
\label{Sec_Hierar}

To introduce the set of tensors in tree-based format with fixed tree-based rank we shall
use the minimal subspaces, in particular, Proposition~2.6 in \cite{FHN} (see also \cite{FALHACK} or \cite{Hackbusch}). Let $\mathcal{P}_D$ be a given partition of $D.$ By definition of the minimal subspaces $U_{\alpha}^{\min}(\mathbf{v}),$ $\alpha \in \mathcal{P}_D$, we have 
$$
\mathbf{v} \in \left. \bigotimes_{\alpha \in \mathcal{P}_D} U_{\alpha}^{\min}(\mathbf{v}%
)\right..$$ For a given $\alpha\in \mathcal{P}_D$ with $\#\alpha \ge 2$ and any partition $\mathcal{P}_{\alpha}$ of $\alpha$, we also have 
\begin{equation*}
 \mathbf{v} \in \left(\left. \bigotimes_{\beta \in \mathcal{P}_{\alpha}}
U_{\beta}^{\min}(\mathbf{v}) \right. \right) \otimes \left(\left.
\bigotimes_{\delta \in \mathcal{P}_D\setminus \{\alpha\}} U_{\delta}^{\min}(\mathbf{v}) \right. \right).
\end{equation*}
Given $D$ we will denote its power set (the set of all subsets of $D$) by $2^D.$ We recall a useful result on the relation between minimal subspaces (see Section 2 in \cite{FALHACK}).
\begin{proposition}\label{inclusin_Umin} 
For any $\alpha \in 2^D$ with $\#\alpha \ge 2$ and any partition $\mathcal{P}_{\alpha}$ of $\alpha$, it holds 
 \begin{equation*}
U_{\alpha}^{\min }(\mathbf{v})\subset \left.
\bigotimes_{\beta \in \mathcal{P}_{\alpha}}U_{\beta}^{\min }(\mathbf{v})\right. .
\end{equation*}
\end{proposition}
 
In order to define tree-based tensor format we introduce three definitions.

\begin{definition}[Dimension partition tree]
\label{partition_tree} A tree $T_{D}$ is called \emph{a dimension
partition tree over $D$} if
\begin{enumerate}
\item[(a)] all vertices $\alpha \in T_D$ are non-empty subsets of $D,$
\item[(b)] $D$ is the \emph{root} of $T_D,$
\item[(c)] every vertex $\alpha \in T_{D}$ with $\#\alpha \geq 2$ has at
least two sons and the set of sons of $\alpha$, denoted $S(\alpha)$, is a non-trivial partition of $\alpha$,
\item[(d)] every vertex $\alpha\in T_D$ with $\#\alpha = 1$ has no son and is called a  \emph{leaf}.
\end{enumerate}
The set of leaves is denoted by $\mathcal{L}(T_{D}).$ 
\end{definition}

A straightforward consequence of Definition~%
\ref{partition_tree} is that the set of leaves $\mathcal{L}(T_{D})$
coincides with the singletons of $D,$ i.e., $\mathcal{L}(T_{D})=\{\{j\}:j\in
D\}$.
\begin{definition}[Levels, depth and partitions]
The \emph{levels} of the vertices of a dimension partition tree $T_D$, denoted by $\mathrm{level}(\alpha)$, $\alpha \in T_D$, are integers defined such that $\mathrm{level}(D) = 0$ and for any pair $\alpha,\beta \in T_D$ such that $\beta \in S(\alpha),$ $\mathrm{level}(\beta) = \mathrm{level}(\alpha)+1$. The \emph{depth} of the tree $T_D$ is defined 
as $\mathrm{depth}(T_D) = \max_{\alpha \in T_D} \level(\alpha)  .$
Then to each level $k$ of $T_D$, $1\le k\le \depth(T_D)$, is associated a
 partition of $D:$
$$
\mathcal{P}_k(T_D) = \{\alpha \in T_D : \mathrm{level}(\alpha) = k\} \cup \{\alpha \in \mathcal{L}(T_D) : \mathrm{level}(\alpha)<k\}.
$$
\end{definition}
\begin{remark}\label{leaf_appears}
Note that for any tree, $\mathcal{P}_{1}(T_D) = S(D)$ and $\mathcal{P}_{\depth(T_D)}(T_D) = \mathcal{L}(T_D)$. Also note that some of the leaves of $T_D$ may be contained in several partitions, and if $\alpha \in \mathcal{L}(T_D)$, then 
$\alpha \in \mathcal{P}_k(T_D)$ for $\level(\alpha) \le k \le \depth(T_D)$. 
\end{remark}
For any partition $\mathcal{P}_k(T_D)$ of level $k$, $1\le k\le \depth(T_D)$, we use the identification $$
\mathbf{V}_D = \left. \bigotimes_{\alpha \in \mathcal{P}_k(T_D)} \mathbf{V}_\alpha \right..
$$
This leads us to the following definition of the representation of the tensor space $\mathbf{V}_D$ in tree-based format.
\begin{definition}
For a tensor space $\mathbf{V}_D$ and a dimension 
partition tree $T_{D}$, the pair 
$(\mathbf{V}_D,T_D)$ is called a representation of the tensor space
$\mathbf{V}_{D}$ in \emph{tree-based format}, and corresponds to the identification   
of $\mathbf{V}_D $ with tensor spaces $ \left. \bigotimes_{\alpha \in \mathcal{P}_k(T_D)} \mathbf{V}_\alpha \right. $ of different levels $k$, $1\le k \le\depth(T_D). $
\end{definition}

\begin{remark}\label{general_chain}
By Proposition~\ref{inclusin_Umin}, for each $\mathbf{v} \in  \mathbf{V}_{D},$ it holds that
$$
\mathbf{v} \in \left. \bigotimes_{\alpha \in \mathcal{P}_1(T_D)}U_{\alpha}^{\min}(\mathbf{v}) \right.
\subset \left. \bigotimes_{\alpha \in \mathcal{P}_2(T_D)}U_{\alpha}^{\min}(\mathbf{v}) \right. \subset
\cdots \subset \left. \bigotimes_{\alpha \in \mathcal{P}_{\depth(T_D)}(T_D)}U_{\alpha}^{\min}(\mathbf{v}) \right..
$$
\end{remark}

\begin{example}[Tucker format]\label{example-tucker}
In Figure~\ref{fig2p}, $D=\{1,2,3,4,5,6\} $ and $$T_D=\{D,\{1\},\{2\},\{3\},\{4\},\{5\},\{6\}\}.$$
Here $\depth(T_D)=1$ and $\mathcal{P}_1(T_D) =  \mathcal{L}(T_D).$ This tree is related to the 
basic identification of $\mathbf{V}_{D}$ with $\left.
\bigotimes_{j=1}^{6}V_{j}\right..$ 
\begin{figure}[h]
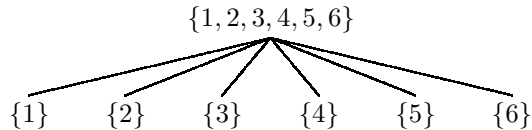

\centering
\synttree[$\{1,2,3,4,5,6\}$[$\{1\}$][$\{2\}$][$\{3\}$]
[$\{4\}$][$\{5\}$][$\{6\}$]]
\caption{a dimension partition tree with $\depth(T_D)=1$ (Tucker tree).}
\label{fig2p}
\end{figure}\end{example}

\begin{example}
\label{example_nobinary} In Figure~\ref{fig1p}, $D=\{1,2,3,4,5,6\}$ and 
\begin{equation*}
T_D=\{D,\{1,2,3\},\{4,5\},\{2,3\},\{1\},\{2\},\{3\},\{4\},\{5\},\{6\}\}.
\end{equation*}
Here $\depth(T_D)=3$, $\mathcal{P}_1(T_D) = \{\{1,2,3\} , \{4,5\} , \{6\}\}$, $\mathcal{P}_2(T_D) = \{\{1\},\{2,3\} , \{4\},\{5\} , \{6\}\}$ and $\mathcal{P}_3(T_D) = \mathcal{L}(T_D).$ This tree is related to the 
 identification of $\mathbf{V}_{D}$ with $\left.
\bigotimes_{j=1}^{6}V_{j}\right.$, $\mathbf{V}_{D} = {V}%
_{1}\otimes \mathbf{V}%
_{23}\otimes {V}_{4}\otimes V_{5}\otimes V_{6} $ and  $\mathbf{V}_{D} = \mathbf{V}%
_{123}\otimes \mathbf{V}_{45}\otimes V_{6} $.
\begin{figure}[h]
\centering
\synttree[$\{1,2,3,4,5,6\}$[$\{1,2,3\}$[$\{1\}$][$\{2,3\}$[$\{2\}$][$\{3\}$]]]
[$\{4,5\}$[$\{4\}$][$\{5\}$]][$\{6\}$ ]]
\caption{A dimension partition tree with $\depth(T_D)=3$.
}\label{fig1p}
\end{figure}
\end{example}
 
Let $\mathbb{N}_{0}:=\mathbb{N} \cup\{0\} $ denote the set of non-negative integers.
For each $\mathbf{v%
} \in \mathbf{V}_D$, we have that $(\dim {U}_{\alpha}^{\min}(\mathbf{v}%
))_{\alpha \in 2^D \setminus \{\emptyset\}}$
is in $\mathbb{N}_0^{2^{\#D}-1}.$ 

\begin{definition}[Tree-based rank]
For a given dimension partition tree $T_D$ over $D$, we define the  
\emph{tree-based rank} 
of a tensor $\mathbf{v}\in
\mathbf{V}_{D}$  by the tuple
$\mathrm{rank}_{T_D}( \mathbf{v} ):= (\dim {U}_{\alpha }^{\min }(%
\mathbf{v}))_{\alpha \in T_{D}}\in \mathbb{N}_0^{\#T_{D}}.$
\end{definition}

\begin{definition}[Admissible ranks]
A tuple $\mathfrak{r}:=(r_{\alpha })_{\alpha \in T_{D}}\in \mathbb{N%
}^{\#T_{D}}$ is said to be an \emph{admissible tuple for $T_{D}$} if there exists $%
\mathbf{v}\in \mathbf{V}_{D}$ such that $\dim
U_{\alpha }^{\min }(\mathbf{v})=r_{\alpha }$ for all $\alpha \in
T_{D}.$ The set of admissible ranks for the representation
$(\mathbf{V}_D,T_D)$ of the tensor space $\mathbf{V}_D$ is denoted by
$$
\mathcal{AD}(\mathbf{V}_D,T_D):=\{(\dim U_{\alpha }^{\min }(\mathbf{v}))_{\alpha \in T_D}:\mathbf{v} \in \mathbf{V}_D\}.
$$
\end{definition}

\begin{definition}
Let $T_{D}$ be a given dimension partition tree and fix some tuple $%
\mathfrak{r}\in \mathcal{AD}(\mathbf{V}_D,T_D)$. Then \emph{the set of 
tensors of fixed tree-based rank $\mathfrak{r}$} is defined by
\begin{equation}
\mathcal{FT}_{\mathfrak{r}}(\mathbf{V}_{D},T_D):=\left\{ \mathbf{v}\in  \mathbf{V}_{D}:\dim {U}_{\alpha }^{\min }(\mathbf{v}%
)=r_{\alpha }\text{ for all }\alpha \in T_{D}\right\}
\end{equation}%
and the \emph{set of tensors of tree-based rank bounded by $\mathfrak{r}$} is
defined by
\begin{equation}
\mathcal{FT}_{\leq \mathfrak{r}}(\mathbf{V}_{D},T_D):=\left\{ \mathbf{v}\in
\mathbf{V}_{D}: 
\dim {U}_{\alpha }^{\min }(\mathbf{v})\leq r_{\alpha }\text{ for all }\alpha
\in T_{D}
\right\} .  \label{(Hr}
\end{equation}
\end{definition}

For $\mathfrak{r},\mathfrak{s}\in \mathbb{%
N}_0^{\#T_{D}}$ we write $\mathfrak{s}\leq \mathfrak{r}$ if and only if $%
s_{\alpha }\leq r_{\alpha }$ for all $\alpha \in T_{D}.$ Then for a fixed
$\mathfrak{r} \in  \mathcal{AD}(\mathbf{V}_D,T_D)$, we have 
\begin{equation}  \label{connected_id}
\mathcal{FT}_{\le \mathfrak{r}}(\mathbf{V}_{D},T_D) :=
\bigcup_{\substack{\mathfrak{s}\leq \mathfrak{r} \\ \mathfrak{s}
\in \mathcal{AD}(\mathbf{V}_D,T_D)}}\mathcal{FT}_{\mathfrak{s}}(\mathbf{V}%
_{D},T_D).
\end{equation}%
For each partition $\mathcal{P}_k(T_D)$ of $D$, $1 \le k \le \depth(T_D)$, we can introduce 
 a set of tensors in
Tucker format with fixed rank $\mathfrak{r}_k:=(r_{\alpha})_{\alpha \in \mathcal{P}_k(T_D)}$ given by
$$
\mathfrak{M}_{\mathfrak{r}_k}(\mathbf{V}_D,\mathcal{P}_k(T_D))
=\{
\mathbf{v} \in \mathbf{V}_D: \dim U_{\alpha}^{\min}(\mathbf{v})=r_{\alpha}
\text{ for } \alpha \in \mathcal{P}_k(T_D)
\}.
$$

\begin{theorem}\label{Char_TBF}
For a dimension partition tree $T_D$ and for $\mathfrak{r}=(r_{\alpha})_{\alpha \in T_D} 
\in  \mathcal{AD}(\mathbf{V}_D,T_D),$ 
$$
\mathcal{FT}_{\mathfrak{r}}(\mathbf{V}_D,T_D) = \bigcap_{k=1}^{\depth(T_D)}
\mathfrak{M}_{\mathfrak{r}_k}(\mathbf{V}_D,\mathcal{P}_k(T_D)).
$$
\end{theorem}

\begin{remark}
We point out that in \cite{FHN} we introduce a representation of $\mathbf{V}_D$ in Tucker format. 
Letting $T_D^{\text{Tucker}}$ be the Tucker dimension partition tree (see example \ref{example-tucker})
and given $\mathfrak{r} \in \mathcal{AD}(\mathbf{V}_D,T_D^{\text{Tucker}})$,   the set of tensors with fixed Tucker rank $\mathfrak{r}$ is defined by
$$
\mathfrak{M}_{\mathfrak{r}}(\mathbf{V}_D):= \mathcal{FT}_{\mathfrak{r}}(\mathbf{V}_{D},T_D^{\text{Tucker}}) =
\left\{\mathbf{v} \in \mathbf{V}_D: \dim U_{k}^{\min}(\mathbf{v}) = r_{k}, \, k \in \mathcal{L}(T_D^{\text{Tucker}}) \right\}.
$$
This leads to the following representation of $\mathbf{V}_D$ in Tucker format:
$$
\mathbf{V}_D = \bigcup_{\mathfrak{r} \in  \mathcal{AD}(\mathbf{V}_D,T_D^{\text{Tucker}})}\mathfrak{M}_{\mathfrak{r}}(\mathbf{V}_D).
$$
Note that for any tree $T_D$ with $\depth(T_D) = 1$,  
$$
\mathfrak{M}_{\mathfrak{r}_{\depth(T_D)}}(\mathbf{V}_D,\mathcal{P}_{\depth(T_D)}(T_D))=
\mathfrak{M}_{\mathfrak{r}_{\depth(T_D)}}(\mathbf{V}_D).
$$
\end{remark}

Finally, we need to take into account the following situation. 
Let  $T_D$ be the rooted tree given in Figure~\ref{fig2pX}. For this rooted
tree we have $\depth(T_D)=2$ and also
\begin{align*}
\mathcal{P}_1(T_D) & =\{\{1\},\{2,3,4,5,6\}\}, \\
\mathcal{P}_2(T_D) & = \{\{1\},\{2\},\{3\},\{4\},\{5\},\{6\}\}.
\end{align*}
From Lemma 2.4 in \cite{FALHACK} it can be shown that $\dim U_{\{1\}}^{\min}(\mathbf{v}) = \dim U_{\{2,3,4,5,6\}}^{\min}(\mathbf{v})$ holds for all $\mathbf{v} \in \mathbf{V}_D.$ Hence
$$
\mathcal{FT}_{\mathfrak{r}}(\mathbf{V}_D,T_D) = \mathfrak{M}_{\mathfrak{r}_1}(\mathbf{V}_D,\mathcal{P}_1(T_D)) \cap \mathfrak{M}_{\mathfrak{r}_2}(\mathbf{V}_D,\mathcal{P}_2(T_D)) = \mathfrak{M}_{\mathfrak{r}_2}(\mathbf{V}_D,\mathcal{P}_2(T_D))
$$
holds because
\begin{align*}
 \mathfrak{M}_{\mathfrak{r}_1}(\mathbf{V}_D,\mathcal{P}_1(T_D)) & =\{
\mathbf{v} \in \mathbf{V}_D: \dim U_{\{1\}}^{\min}(\mathbf{v}) = r_{\{1\}} = \dim U_{\{2,3,4,5,6\}}^{\min}(\mathbf{v})
\}  
\end{align*}
contains
\begin{align*}
\mathfrak{M}_{\mathfrak{r}_2}(\mathbf{V}_D,\mathcal{P}_2(T_D)) & =\{
\mathbf{v} \in \mathbf{V}_D: \dim U_{\{i\}}^{\min}(\mathbf{v}) = r_{\{i\}}, \, 1 \le i \le 6
\} .
\end{align*}
\begin{figure}[h]
\centering
\synttree[$\{1,2,3,4,5,6\}$[$\{1\}$][$\{2,3,4,5,6\}$[$\{2\}$][$\{3\}$]
[$\{4\}$][$\{5\}$][$\{6\}$ ]]]
\caption{A dimension partition tree with $\depth(T_D)=2.$}
\label{fig2pX}
\end{figure}
Thus in order to avoid this situation we introduce the following definition. 
 
\begin{definition}
For a dimension partition tree $T_D$ and for $\mathfrak{r}=(r_{\alpha})_{\alpha \in T_D} 
\in  \mathcal{AD}(\mathbf{V}_D,T_D),$ we will say that $\mathcal{FT}_{\mathfrak{r}}(\mathbf{V}_D,T_D)$ is a
proper set of tree-based tensors with a fixed tree-based rank $\mathfrak{r}$ if
$$
\mathcal{FT}_{\mathfrak{r}}(\mathbf{V}_D,T_D) \neq \mathfrak{M}_{\mathfrak{r}_k}(\mathbf{V}_D,\mathcal{P}_k(T_D)) \text{ holds for } 1 \le k \le \depth(T_D).
$$ 
\end{definition}

\section{The manifold of tensors in Tucker format with fixed rank}\label{Tucker}

In this section we start by introducing the geometric structure of the set of tensors in Tucker format with fixed rank in our framework. Next, we give an equivalent result that allows us to provide a manifold structure
to a proper set of tree-based tensors with a fixed tree-based rank.

\bigskip

Assume that $\mathcal{P}_D$ is a partition of $D$ and 
$(\mathbf{V}_{\alpha},\|\cdot\|_{\alpha})$ is a normed space for each $\alpha \in \mathcal{P}_D.$ 
We will consider the product space $\mathop{\mathchoice{\raise-0.22em\hbox{\huge $\times$}}
{\raise-0.05em\hbox{\Large $\times$}}{\hbox{\large $\times$}}{\times}}%
_{\alpha \in \mathcal{P}_D} \mathbf{V}_{\alpha}$ equipped with the
product topology induced by the maximum norm $\|(\mathbf{v}_{\alpha})_{\alpha \in \mathcal{P}_D})\|_{\times} =
\max_{\alpha \in \mathcal{P}_D} \|\mathbf{v}_{\alpha}\|_{\alpha}.$
Then, from Theorem~3.17 in \cite{FHN}, we have the following result.

\begin{theorem}
\label{Tucker_Banach_Manifold} Assume that $\mathcal{P}_D$ is a partition of $D,$ 
$(\mathbf{V}_{\alpha},\|\cdot\|_{\alpha})$ is a normed space for each $\alpha \in \mathcal{P}_D$ 
and that $\|\cdot\|_D$ is a norm 
on the tensor space $\mathbf{V}_D = \left.
\bigotimes_{\alpha \in \mathcal{P}_D}\mathbf{V}_{\alpha}\right.$ such that the tensor
product map  
\begin{equation}
\bigotimes_{\alpha \in \mathcal{P}_D} :\left(\mathop{\mathchoice{\raise-0.22em\hbox{\huge $\times$}}
{\raise-0.05em\hbox{\Large $\times$}}{\hbox{\large $\times$}}{\times}}%
_{\alpha \in \mathcal{P}_D} \mathbf{V}_{\alpha},\left\Vert \cdot\right\Vert_{\times} \right) 
\longrightarrow%
\bigg(
\left. \bigotimes_{\alpha \in \mathcal{P}_D}\mathbf{V}_{\alpha}\right. ,\left\Vert \cdot\right\Vert_D
\bigg)
,  \label{bigotimes}
\end{equation}
is continuous.
Then there exists a  $\mathcal{C}^{\infty}$-atlas $\{\mathcal{U}(\mathbf{v}),\widetilde{\xi}_{\mathbf{v}}\}_{\mathbf{v}\in \mathfrak{M}_{\mathfrak{%
r}}(\mathbf{V}_{D})}$ for $\mathfrak{M}_{\mathfrak{r}}(%
\mathbf{V}_{D})$ and hence $\mathfrak{M}_{\mathfrak{r}}(%
\mathbf{V}_{D})$ is a $\mathcal{C}^{\infty}$-Banach manifold modelled on a Banach space
$$
\left( 
\mathop{\mathchoice{\raise-0.22em\hbox{\huge
$\times$}} {\raise-0.05em\hbox{\Large $\times$}}{\hbox{\large
$\times$}}{\times}}_{\alpha \in \mathcal{P}_D}\mathcal{L}(\mathbf{U}_{\alpha},\mathbf{W}_{\alpha})\right) \times \mathbb{R}^{
\mathop{\mathchoice{\raise-0.22em\hbox{\huge $\times$}}
{\raise-0.05em\hbox{\Large $\times$}}{\hbox{\large
$\times$}}{\times}}_{\alpha \in \mathcal{P}_D}r_{\alpha}}.
$$
Here $\mathbf{U}_{\alpha} \in \mathbb{G}_{r_{\alpha}}(\mathbf{V}_{\alpha})$ and $\mathbf{V}_{\alpha_{\|\cdot\|_{\alpha}}} = \mathbf{U}_{\alpha} \oplus \mathbf{W}_{\alpha},$ where $\mathbf{V}_{{\alpha}_{\|\cdot\|_{\alpha}}}$ is the completion of
$\mathbf{V}_{\alpha}$ for $\alpha \in \mathcal{P}_D.$
\end{theorem}

To define a manifold structure (see \cite{Lang}) we did not require that the vector 
spaces involved as coordinates are the same or even linearly isomorphic. In our case, we have that
$U_{\alpha}^{\min}(\mathbf{v})$ is linearly isomorphic to $U_{\alpha}^{\min}(\mathbf{w})$ 
for all $\mathbf{w} \in \mathfrak{M}_{\mathfrak{r}}(\mathbf{V}_{D})$. Thus, we fix one
$\mathbf{U}_{\alpha} = U_{\alpha}^{\min}(\mathbf{v})$ and hence it can be shown that
$\mathcal{L}(\mathbf{U}_{\alpha},\mathbf{W}_{\alpha})$ is linearly isomorphic to 
$\mathcal{L}(U_{\alpha}^{\min}(\mathbf{w}),W_{\alpha}^{\min}(\mathbf{w}))$ for all $\mathbf{w} \in \mathfrak{M}_{\mathfrak{r}}(\mathbf{V}_{D}),$ where 
$W_{\alpha}^{\min}(\mathbf{w})$ is linearly isomorphic to $\mathbf{W}_{\alpha}$ and
satisfies $\mathbf{V}_{{\alpha}_{\|\cdot\|_{\alpha}}} = U_{\alpha}^{\min}(\mathbf{w}) \oplus W_{\alpha}^{\min}(\mathbf{w}).$ Moreover,
$\left. \bigotimes_{\alpha \in \mathcal{P}_D}\mathbf{U}_{\alpha} \right.$ is linearly isomorphic
to $\left. \bigotimes_{\alpha \in \mathcal{P}_D}U_{\alpha}^{\min}(\mathbf{w}) \right.$ for all $\mathbf{w} \in \mathfrak{M}_{\mathfrak{r}}(\mathbf{V}_{D}).$ In consequence, $\mathfrak{M}_{\mathfrak{r}}(\mathbf{V}_{D})$ has a geometric structure modelled on the Banach space
$$
\left(\mathop{\mathchoice{\raise-0.22em\hbox{\huge
$\times$}} {\raise-0.05em\hbox{\Large $\times$}}{\hbox{\large
$\times$}}{\times}}_{\alpha \in \mathcal{P}_D}\mathcal{L}(\mathbf{U}_{\alpha},\mathbf{W}_{\alpha}) \right)\times \left. \bigotimes_{\alpha \in \mathcal{P}_D}\mathbf{U}_{\alpha} \right.,
$$ 
which is linearly isomorphic to
$$
\left( 
\mathop{\mathchoice{\raise-0.22em\hbox{\huge
$\times$}} {\raise-0.05em\hbox{\Large $\times$}}{\hbox{\large
$\times$}}{\times}}_{\alpha \in \mathcal{P}_D}\mathcal{L}(\mathbf{U}_{\alpha},\mathbf{W}_{\alpha})\right) \times \mathbb{R}^{
\mathop{\mathchoice{\raise-0.22em\hbox{\huge $\times$}}
{\raise-0.05em\hbox{\Large $\times$}}{\hbox{\large
$\times$}}{\times}}_{\alpha \in \mathcal{P}_D}r_{\alpha}}.
$$

\bigskip

The atlas $\{\mathcal{U}(\mathbf{v}),\widetilde{\xi}_{\mathbf{v}}\}_{\mathbf{v}\in \mathfrak{M}_{\mathfrak{%
r}}(\mathbf{V}_{D})}$ from Theorem \ref{Tucker_Banach_Manifold}  is composed by a subset $\mathcal{U}(\mathbf{v}) \subset \mathfrak{M}_{\mathfrak{r}}(\mathbf{V}_{D})$ containing $\mathbf{v}$ and a bijection $\widetilde{\xi}_{\mathbf{v}}$ from $\mathcal{U}(\mathbf{v})$
to the open set 
$$
\left( 
\mathop{\mathchoice{\raise-0.22em\hbox{\huge
$\times$}} {\raise-0.05em\hbox{\Large $\times$}}{\hbox{\large
$\times$}}{\times}}_{\alpha \in \mathcal{P}_D}\mathcal{L}(\mathbf{U}_{\alpha},\mathbf{W}_{\alpha})\right) \times 
\mathfrak{M}_{\mathfrak{r}}\left(\left. \bigotimes_{\alpha \in \mathcal{P}_D}\mathbf{U}_{\alpha} \right.\right),
$$
which is contained in the Banach space
$$
\left( 
\mathop{\mathchoice{\raise-0.22em\hbox{\huge
$\times$}} {\raise-0.05em\hbox{\Large $\times$}}{\hbox{\large
$\times$}}{\times}}_{\alpha \in \mathcal{P}_D}\mathcal{L}(\mathbf{U}_{\alpha},\mathbf{W}_{\alpha})\right) \times 
\left. \bigotimes_{\alpha \in \mathcal{P}_D}\mathbf{U}_{\alpha} \right..
$$
From Lemma 3.12 in \cite{FHN}, for $\mathbf{w} \in \mathcal{U}(\mathbf{v}),$ we have $\widetilde{\xi}_{\mathbf{v}}(\mathbf{w}) = ((L_{\alpha})_{\alpha \in \mathcal{P}_D},\mathbf{u})$ if and
only if $$\mathbf{w} =\widetilde{\xi}_{\mathbf{v}}^{-1}((L_{\alpha})_{\alpha \in \mathcal{P}_D},\mathbf{u})  = \left(\bigotimes_{\alpha \in \mathcal{P}_D} \exp(L_{\alpha})\right)(\mathbf{u}).$$ In particular, we have $\widetilde{\xi}_{\mathbf{v}}(\mathbf{v}) = ((0_{\alpha})_{\alpha \in \mathcal{P}_D},\mathbf{v}),$ where
$0_{\alpha}$ denotes the zero map in $\mathcal{L}(\mathbf{U}_{\alpha},\mathbf{W}_{\alpha}).$

\bigskip

We recall that $\overline{\mathbf{V}_D}^{\|\cdot\|_D}
= \mathbf{V}_{D_{\|\cdot\|_D}}$ denotes the tensor Banach space obtained 
as the completion of the algebraic
tensor space $\mathbf{V}_D$  under the norm $\|\cdot\|_D.$ 
In the case where $\mathbf{V}_D$ is finite dimensional, $  \mathbf{V}_{D_{\|\cdot\|_D}} = \mathbf{V}_D$. Otherwise, $  \mathbf{V}_D \subsetneq  \mathbf{V}_{D_{\|\cdot\|_D}}$.
Our next step is, given a fixed partition $\mathcal{P}_D$ of $D,$ 
to identify the Banach space $\bigtimes_{\alpha \in \mathcal{P}_D}\mathcal{L}(\mathbf{U}_\alpha ,\mathbf{W}_\alpha)$ with
a closed subspace of the Banach algebra $\mathcal{L}(\mathbf{V}_{D_{\|\cdot\|_D}},\mathbf{V}_{D_{\|\cdot\|_D}}).$ To this end, we need to proceed in the framework of Section 4 in \cite{FHN}. First, we recall the definition of injective norm (Definition 4.9 in \cite{FHN}) stated in the present framework.

\begin{definition}
Let $\mathbf{V}_{\alpha}$ be a Banach space with norm $\left\Vert \cdot\right\Vert _{\alpha}$
for $\alpha\in \mathcal{P}_D.$ Then for $\mathbf{v}\in\mathbf{V}=\left.
\bigotimes_{\alpha \in \mathcal{P}_D}\mathbf{V}_{\alpha}\right. $ define $\left\Vert \cdot\right\Vert
_{\vee((\mathbf{V}_{\alpha})_{\alpha \in \mathcal{P}_D})}$ by%
\begin{equation}
\left\Vert \mathbf{v}\right\Vert _{\vee((\mathbf{V}_{\alpha})_{\alpha \in \mathcal{P}_D})}:=\sup\left\{ \frac{%
\left\vert \left(
\bigotimes_{\alpha \in \mathcal{P}_D}\varphi_{\alpha}\right) (\mathbf{%
v})\right\vert }{\prod_{\alpha \in \mathcal{P}_D}\Vert\varphi_{\alpha}\Vert_{\alpha}^{\ast}}%
:0\neq\varphi_{\alpha}\in \mathbf{V}_{\alpha}^{\ast},\alpha \in \mathcal{P}_D\right\} ,
\label{(Norm ind*(V1,...,Vd)}
\end{equation}
where $\mathbf{V}_{\alpha}^{\ast}$ is the continuous dual of $\mathbf{V}_{\alpha}$.
\end{definition}

Let $W$ and $U$  be closed subspaces of a Banach space $X$ such that $%
X=U\oplus W.$ From now on, we will denote by $P_{_{U\oplus W}}$ the
projection onto $U$ along $W.$ 
Then we have $P_{_{W\oplus U}}=id_{X}-P_{_{U\oplus W}}.$ The proof of the 
next result uses Proposition 2.8, Lemma 4.13 and Lemma 4.14 in \cite{FHN}.

\begin{lemma}\label{G1}
Assume that
$(\mathbf{V}_{\alpha},\|\cdot\|_{\alpha})$ is a normed space 
for each $\alpha \in \mathcal{P}_D$ 
and let $\|\cdot\|_D$ be a norm 
on the tensor space $\mathbf{V}_D = \left.
\bigotimes_{\alpha \in \mathcal{P}_D}\mathbf{V}_{\alpha}\right.$ such that 
\begin{align}\label{tree_injective_norm}
\left\Vert \cdot\right\Vert
_{\vee((\mathbf{V}_{\alpha})_{\alpha \in \mathcal{P}_D})}\lesssim\left\Vert \cdot\right\Vert_D, 
\end{align} 
holds. Let $\mathbf{U}_{\alpha} \in \mathbb{G}_{r_{\alpha}}(\mathbf{V}_{\alpha})$ and $\mathbf{V}_{\alpha_{\|\cdot\|_{\alpha}}} = \mathbf{U}_{\alpha} \oplus \mathbf{W}_{\alpha},$ where $\mathbf{V}_{{\alpha}_{\|\cdot\|_{\alpha}}}$ is the completion of
$\mathbf{V}_{\alpha}$ for $\alpha \in \mathcal{P}_D.$ Then for each $\alpha \in \mathcal{P}_D$ we have
$$
\mathcal{L}\left(\mathbf{U}_{\alpha} ,\mathbf{W}_{\alpha} \right) \otimes
\mathrm{span}\{\mathbf{id}_{[\alpha]}\}  \in \mathbb{G}\left( \mathcal{L}(\mathbf{V}_{D_{\|\cdot\|_D}},\mathbf{V}_{D_{\|\cdot\|_D}})\right)
$$
where  $\mathbf{id}_{[\alpha]} := \bigotimes_{\beta \in \mathcal{P}_D \setminus \{\alpha\}} id_{\mathbf{V}_{\beta}}.$
Furthermore,
$$
\bigoplus_{\alpha \in \mathcal{P}_D}\mathcal{L}\left(\mathbf{U}_{\alpha} ,\mathbf{W}_{\alpha} \right) \otimes
\mathrm{span}\{\mathbf{id}_{[\alpha]}\}  \in \mathbb{G}\left( \mathcal{L}(\mathbf{V}_{D_{\|\cdot\|_D}},\mathbf{V}_{D_{\|\cdot\|_D}})\right).
$$
\end{lemma}

\begin{proof}
To prove the lemma, for a fixed $\alpha \in \mathcal{P}_D,$ note that
$id_{\mathbf{V}_{\alpha}} = P_{_{\mathbf{U}_{\alpha} \oplus \mathbf{W}_{\alpha}}}+ P_{_{\mathbf{W}_{\alpha}\oplus \mathbf{U}_{\alpha} }}$
and write
$$
id_{\mathbf{V}_{D_{\|\cdot\|_D}}} = id_{\mathbf{V}_{\alpha}} \otimes \mathbf{id}_{[\alpha]}.
$$
Since $\mathbf{U}_{\alpha}$ is a finite dimensional space, $P_{_{\mathbf{U}_{\alpha}\oplus \mathbf{W}_{\alpha} }}$ is a finite rank projection and hence $P_{_{\mathbf{U}_{\alpha}\oplus \mathbf{W}_{\alpha} }} \otimes \mathbf{id}_{[\alpha]} \in \mathcal{L}(\mathbf{V}_{D_{\|\cdot\|_D}},\mathbf{V}_{D_{\|\cdot\|_D}}).$ Then by proceeding as in the proof of Lemma 4.13 in \cite{FHN} we obtain that
$$
P_{_{\mathbf{W}_{\alpha}\oplus \mathbf{U}_{\alpha} }} \otimes \mathbf{id}_{[\alpha]} \in \mathcal{L}(\mathbf{V}_{D_{\|\cdot\|_D}},\mathbf{V}_{D_{\|\cdot\|_D}}).
$$
Now, define the linear and bounded map
$$
\mathcal{P}_{\alpha}:\mathcal{L}(\mathbf{V}_{D_{\|\cdot\|_D}},\mathbf{V}_{D_{\|\cdot\|_D}}) \longrightarrow
\mathcal{L}(\mathbf{V}_{D_{\|\cdot\|_D}},\mathbf{V}_{D_{\|\cdot\|_D}})
$$
as $\mathcal{P}_{\alpha}(L) = (P_{_{\mathbf{W}_{\alpha}\oplus \mathbf{U}_{\alpha} }} \otimes \mathbf{id}_{[\alpha]}) \circ L \circ (P_{_{\mathbf{U}_{\alpha}\oplus \mathbf{W}_{\alpha} }} \otimes \mathbf{id}_{[\alpha]}).$ It satisfies 
$\mathcal{P}_{\alpha} \circ \mathcal{P}_{\alpha} = \mathcal{P}_{\alpha}$ and
$$
\mathcal{P}_{\alpha}(\mathcal{L}(\mathbf{V}_{D_{\|\cdot\|_D}},\mathbf{V}_{D_{\|\cdot\|_D}})) =
\mathcal{L}\left(\mathbf{U}_{\alpha} ,\mathbf{W}_{\alpha} \right) \otimes
\mathrm{span}\{\mathbf{id}_{[\alpha]}\} .
$$
Proposition~2.8(b) in \cite{FHN} implies that $\mathcal{L}\left(\mathbf{U}_{\alpha} ,\mathbf{W}_{\alpha} \right) \otimes
\mathrm{span}\{\mathbf{id}_{[\alpha]}\}  \in \mathbb{G}\left( \mathcal{L}(\mathbf{V}_{D_{\|\cdot\|_D}},\mathbf{V}_{D_{\|\cdot\|_D}})\right).$
Observe that for $\alpha,\beta \in \mathcal{P}_D$ with $\alpha \neq \beta$ we have
$$
(\mathcal{L}\left(\mathbf{U}_{\alpha} ,\mathbf{W}_{\alpha} \right) \otimes
\mathrm{span}\{\mathbf{id}_{[\alpha]}\} ) \cap (\mathcal{L}\left(\mathbf{U}_{\beta} ,\mathbf{W}_{\beta} \right) \otimes
\mathrm{span}\{\mathbf{id}_{[\beta]}\}) = \{\mathbf{0}\}.
$$
By Lemma~4.14 in \cite{FHN} we have
$$
\bigoplus_{\alpha \in \mathcal{P}_D}\mathcal{L}\left(\mathbf{U}_{\alpha} ,\mathbf{W}_{\alpha} \right) \otimes
\mathrm{span}\{\mathbf{id}_{[\alpha]}\}  \in \mathbb{G}\left( \mathcal{L}(\mathbf{V}_{D_{\|\cdot\|_D}},\mathbf{V}_{D_{\|\cdot\|_D}})\right).
$$
This proves the lemma.
\end{proof}

\bigskip

Lemma~\ref{G1} allows to introduce the following linear isomorphism:
$$
\Delta:\mathop{\mathchoice{\raise-0.22em\hbox{\huge
 $\times$}} {\raise-0.05em\hbox{\Large $\times$}}{\hbox{\large
 $\times$}}{\times}}_{\alpha \in  \mathcal{P}_D}\mathcal{L}(\mathbf{U}_{\alpha},\mathbf{W}_{\alpha}) \longrightarrow 
\bigoplus_{\alpha \in \mathcal{P}_D}\mathcal{L}\left(U_{\alpha} ,\mathbf{W}_{\alpha} \right) \otimes
\mathrm{span}\{\mathbf{id}_{[\alpha]}\} 
 , \quad 
 (L_{\alpha})_{\alpha \in  \mathcal{P}_D} \mapsto \sum_{\alpha \in  \mathcal{P}_D} L_{\alpha} \otimes \mathbf{id}_{[\alpha]}.
$$
where $\mathbf{id}_{[\alpha]} := \bigotimes_{\beta \in \mathcal{P}_D \setminus \{\alpha\}} id_{\mathbf{V}_{\beta}}$
for $\alpha \in \mathcal{P}_D.$ 
The next proposition gives us a useful property of the elements in the image of the map $\Delta.$

\begin{proposition}\label{fundamental_property} Assume that $\mathcal{P}_D$ is a partition of $D,$
 $(\mathbf{V}_{\alpha},\|\cdot\|_{\alpha})$ is a normed space for each $\alpha \in \mathcal{P}_D$ 
and  $\|\cdot\|_D$ is a norm 
on the tensor space $\mathbf{V}_D = \left.
\bigotimes_{\alpha \in \mathcal{P}_D} \mathbf{V}_{\alpha}\right.$ such that (\ref{tree_injective_norm}) holds.  Then for each 
$(L_{\alpha})_{\alpha \in \mathcal{P}_D} \in \mathop{\mathchoice{\raise-0.22em\hbox{\huge
 $\times$}} {\raise-0.05em\hbox{\Large $\times$}}{\hbox{\large
 $\times$}}{\times}}_{\alpha \in \mathcal{P}_D}\mathcal{L}(\mathbf{U}_{\alpha},\mathbf{W}_{\alpha})$ it holds that
$$
\exp(\Delta\left((L_{\alpha})_{\alpha \in \mathcal{P}_D})\right)= 
\bigotimes_{\alpha \in \mathcal{P}_D} \exp(L_{\alpha}).
$$
\end{proposition}

\begin{proof}
 Put $L:=\Delta\left((L_{\alpha})_{\alpha \in \mathcal{P}_D}\right) = \sum_{\alpha \in \mathcal{P}_D}L_{\alpha} \otimes \mathbf{id}_{[\alpha]}$ 
 and observe that for each $\alpha \in \mathcal{P}_D$ it
holds 
$$
\exp(L_{\alpha} \otimes \mathbf{id}_{[\alpha]}) = \sum_{n=0}^{\infty} \frac{1}{n!}
(L_{\alpha} \otimes \mathbf{id}_{[\alpha]})^n =  
\left(\sum_{n=0}^{\infty} \frac{1}{n!} L_{\alpha}^n\right) \otimes \mathbf{id}_{[\alpha]}
= \exp(L_{\alpha})\otimes \mathbf{id}_{[\alpha]}.
$$
Moreover for $\alpha,\beta \in \mathcal{P}_D$ and $\alpha \neq \beta$ we have
$$
(L_{\alpha} \otimes \mathbf{id}_{[\alpha]}) \circ (L_{\beta} \otimes \mathbf{id}_{[\beta]})
= (L_{\beta} \otimes \mathbf{id}_{[\beta]}) \circ (L_{\alpha} \otimes \mathbf{id}_{[\alpha]})
= L_{\alpha} \otimes L_{\beta} \otimes \left(\bigotimes_{\delta \in \mathcal{P}_D \setminus \{\alpha , \beta\}} id_{\mathbf{V}_{\delta}}\right).
$$
Finally, by seing $\mathcal{P}_D$ as an ordered set, and by denoting $
\bigodot_{i=1}^n A_i := A_1 \circ A_2 \circ \cdots \circ A_n
$ is the composition of maps $A_i$, $1\le i \le n$, we have
\begin{align*}
\exp\left( L \right)=
\bigodot_{\alpha \in \mathcal{P}_D} \exp(L_{\alpha} \otimes \mathbf{id}_{[\alpha]})  =
\bigodot_{\alpha \in \mathcal{P}_D} \exp(L_{\alpha})\otimes \mathbf{id}_{[\alpha]} = 
 \bigotimes_{\alpha \in \mathcal{P}_D} \exp(L_{\alpha}).
\end{align*}
Note that since operators $\exp(L_{\alpha})\otimes \mathbf{id}_{[\alpha]}$ and $\exp(L_{\beta})\otimes \mathbf{id}_{[\beta]}$  commute for any $\alpha,\beta \in \mathcal{P}_D$, the above result is independent of the chosen order on $\mathcal{P}_D$. 
This proves the proposition.
\end{proof}

\bigskip

To simplify notation, let  
$$
\mathbf{E}_{\mathcal{P}_D}:= \left(\bigoplus_{\alpha \in \mathcal{P}_D}\mathcal{L}\left(\mathbf{U}_{\alpha} ,\mathbf{W}_{\alpha} \right) \otimes
\mathrm{span}\{\mathbf{id}_{[\alpha]}\}\right).
$$
Recall that $\widetilde{\xi}_{\mathbf{v}}$ is a bijection from $\mathcal{U}(\mathbf{v})$ to the open set
$$
\left( 
\mathop{\mathchoice{\raise-0.22em\hbox{\huge
$\times$}} {\raise-0.05em\hbox{\Large $\times$}}{\hbox{\large
$\times$}}{\times}}_{\alpha \in \mathcal{P}_D}\mathcal{L}(\mathbf{U}_{\alpha},\mathbf{W}_{\alpha})\right) \times 
\mathfrak{M}_{\mathfrak{r}}\left(\left. \bigotimes_{\alpha \in \mathcal{P}_D}\mathbf{U}_{\alpha} \right.\right).
$$
Hence the map $\xi_{\mathbf{v}}:= (\Delta \times id) \circ \widetilde{\xi}_{\mathbf{v}},$ where $id: \mathfrak{M}_{\mathfrak{r}}\left(\left. \bigotimes_{\alpha \in \mathcal{P}_D}\mathbf{U}_{\alpha} \right.\right) \longrightarrow \mathfrak{M}_{\mathfrak{r}}\left(\left. \bigotimes_{\alpha \in \mathcal{P}_D}\mathbf{U}_{\alpha} \right.\right)$ is the identity map, is
a bijection from $\mathcal{U}(\mathbf{v})$ to the open set
$$
\mathbf{E}_{\mathcal{P}_D}  \times 
\mathfrak{M}_{\mathfrak{r}}\left(\left. \bigotimes_{\alpha \in \mathcal{P}_D}\mathbf{U}_{\alpha} \right.\right).
$$
For each $\mathbf{w} \in \mathcal{U}(\mathbf{v}),$ we have $ \widetilde \xi_{\mathbf{v}}(\mathbf{w}) = ((L_{\alpha})_{\alpha \in \mathcal{P}_D},\mathbf{u})$ for some 
$(L_{\alpha})_{\alpha \in \mathcal{P}_D} \in \left( 
\mathop{\mathchoice{\raise-0.22em\hbox{\huge
$\times$}} {\raise-0.05em\hbox{\Large $\times$}}{\hbox{\large
$\times$}}{\times}}_{\alpha \in \mathcal{P}_D}\mathcal{L}(\mathbf{U}_{\alpha},\mathbf{W}_{\alpha})\right)  $ and $\mathbf{u} \in \mathfrak{M}_{\mathfrak{r}}\left(\left. \bigotimes_{\alpha \in \mathcal{P}_D}\mathbf{U}_{\alpha} \right.\right)$. Then, letting $L := \Delta((L_{\alpha})_{\alpha \in \mathcal{P}_D}),$ 
$$
\mathbf{w} = \xi_{\mathbf{v}}^{-1}(L, \mathbf{u}) =
 \xi_{\mathbf{v}}^{-1}(\Delta((L_{\alpha})_{\alpha \in \mathcal{P}_D}), \mathbf{u})
= ((\Delta \times id) \circ \widetilde{\xi}_{\mathbf{v}})^{-1}(\Delta((L_{\alpha})_{\alpha \in \mathcal{P}_D}), \mathbf{u})
= \widetilde{\xi}_{\mathbf{v}}^{-1}((L_{\alpha})_{\alpha \in \mathcal{P}_D},\mathbf{u}).
$$
Thus, thanks to Proposition~\ref{fundamental_property}, we deduce that the equality 
\begin{align}\label{Eq1}
\mathbf{w} = \xi_{\mathbf{v}}^{-1}(\Delta((L_{\alpha})_{\alpha \in \mathcal{P}_D}), \mathbf{u}) =  \left( \bigotimes_{
\alpha \in \mathcal{P}_D}\exp(L_{\alpha})\right) (\mathbf{u})
\end{align}
is equivalent to
\begin{align}\label{igualdad_fundamental}
\mathbf{w} = \xi_{\mathbf{v}}^{-1}(L, \mathbf{u}) =  \exp(L)(\mathbf{u}),
\end{align}
where $L= \sum_{\alpha \in \mathcal{P}_D}L_{\alpha} \otimes \mathbf{id}_{[\alpha]}$ is a Laplacian-like map. In consequence, every tensor in Tucker format is locally characterised by a full-rank tensor and a Laplacian-like map. 
To conclude, we can re-state Theorem~\ref{Tucker_Banach_Manifold} as follows.

\begin{theorem}
\label{Tucker_Banach_ManifoldXXX} 
Assume that $\mathcal{P}_D$ is a partition of $D,$ 
$(V_{\alpha},\|\cdot\|_{\alpha})$ is a normed space 
for each $\alpha \in \mathcal{P}_D$ 
and let $\|\cdot\|_D$ be a norm 
on the tensor space $\mathbf{V}_D = \left.
\bigotimes_{\alpha \in D}\mathbf{V}_{\alpha}\right.$ such that  
\eqref{tree_injective_norm} holds. 
Then there exists a  $\mathcal{C}^{\infty}$-atlas $\{\mathcal{U}(\mathbf{v}),\xi_{\mathbf{v}}\}_{\mathbf{v}\in \mathfrak{M}_{\mathfrak{%
r}}(\mathbf{V}_{D})}$ for $\mathfrak{M}_{\mathfrak{r}}(%
\mathbf{V}_{D})$ and hence $\mathfrak{M}_{\mathfrak{r}}(%
\mathbf{V}_{D})$ is a $\mathcal{C}^{\infty}$-Banach manifold modelled on a Banach space
$$
\mathbf{E}_{\mathcal{P}_D} \times \mathbb{R}^{
\mathop{\mathchoice{\raise-0.22em\hbox{\huge $\times$}}
{\raise-0.05em\hbox{\Large $\times$}}{\hbox{\large
$\times$}}{\times}}_{\alpha \in \mathcal{P}_D}r_{\alpha}},
$$
here $\mathbf{U}_{\alpha}\in \mathbb{G}_{r_{\alpha}}(\mathbf{V}_{\alpha})$ and $\mathbf{V}_{\alpha_{\|\cdot\|_{\alpha}}} = \mathbf{U}_{\alpha} \oplus \mathbf{W}_{\alpha},$ where $\mathbf{V}_{{\alpha}_{\|\cdot\|_{\alpha}}}$ is the completion of
$\mathbf{V}_{\alpha}$ for $\alpha \in \mathcal{P}_D.$
\end{theorem}

Observe that for any partition $\mathcal{P}_D$ of $D,$ from Lemma~\ref{G1}, 
the Banach space $\mathbf{E}_{\mathcal{P}_D}$ is a closed linear subspace of the Banach space $\mathcal{L}(\mathbf{V}_{D_{\|\cdot\|_D}},\mathbf{V}_{D_{\|\cdot\|_D}}).$

\section{The geometry of tree-based tensor format}
\label{sec:banach_manifold_tucker_fixed_rank}
For a dimension partition tree $T_D$ and for $\mathfrak{r}=(r_{\alpha})_{\alpha \in T_D} 
\in  \mathcal{AD}(\mathbf{V}_D,T_D),$ assume that $\mathcal{FT}_{\mathfrak{r}}(\mathbf{V}_D,T_D)$ is a
proper set of tree-based tensors with a fixed tree-based rank $\mathfrak{r}$ such
that
$$
\mathcal{FT}_{\mathfrak{r}}(\mathbf{V}_D,T_D) = \bigcap_{k=1}^{\depth(T_D)}
\mathfrak{M}_{\mathfrak{r}_k}(\mathbf{V}_D,\mathcal{P}_k(T_D)).
$$
Assume that
$(\mathbf{V}_{\alpha},\|\cdot\|_{\alpha})$ is a normed space for each $\alpha \in \mathcal{P}_{k}(T_D)$ and that $\|\cdot\|_D$ is a norm 
on the tensor space $\mathbf{V}_D = \left.
\bigotimes_{\alpha \in \mathcal{P}_{k}(T_D)} \mathbf{V}_{\alpha}\right.$ such that 
\eqref{tree_injective_norm} holds for $1 \le k \le \depth(T_D).$  

\bigskip

From Theorem~\ref{Tucker_Banach_ManifoldXXX} we have that for each $1 \le k \le \depth(T_D)$ the collection $\mathcal{A}_{k} = \{(\mathcal{U}^{(k)}(\mathbf{v}),\xi_{\mathbf{v}}^{(k)})\}_{\mathbf{v}\in \mathfrak{M}_{\mathfrak{%
r}_{k}}(\mathbf{V}_{D},T_D)}$ is a  $\mathcal{C}^{\infty}$-atlas for $\mathfrak{M}_{\mathfrak{r}_{k}}(%
\mathbf{V}_{D},T_D)$ and hence  $\mathfrak{M}_{\mathfrak{r}_{{k}}}(%
\mathbf{V}_{D},T_D)$  is a $\mathcal{C}^{\infty}$-Banach manifold modelled on 
$$
\left(\bigoplus_{\alpha \in \mathcal{P}_k(T_D)}\mathcal{L}\left(\mathbf{U}_{\alpha} ,\mathbf{W}_{\alpha} \right) \otimes
\mathrm{span}\{\mathbf{id}_{[\alpha]}\}\right) \times \mathbb{R}^{
\bigtimes_{\alpha \in \mathcal{P}_{k}(T_D)}r_{\alpha}},
$$
where $ \mathbf{U}_{\alpha}=U_{\alpha}^{\min}(\mathbf{v})$ is a $r_{\alpha}$-dimensional subspace of $\mathbf{V}_{\alpha}$ for each $\alpha \in \mathcal{P}_k(T_D)$ where $\mathbf{v} \in  \left. \bigotimes_{\alpha \in \mathcal{P}_k(T_D)}\mathbf{U}_{\alpha} \right.$ and $\mathbf{W}_\alpha$ is a closed subspace of $\mathbf{V}_{\alpha_{\|\cdot\|_{\alpha}}}$ such that $\mathbf{V}_{\alpha_{\|\cdot\|_{\alpha}}} = \mathbf{U}_{\alpha} \oplus \mathbf{W}_{\alpha} ,$ where $\mathbf{V}_{{\alpha}_{\|\cdot\|_{\alpha}}}$ is the completion of
$\mathbf{V}_{\alpha}$ for $\alpha \in \mathcal{P}_{k}(T_D).$

\bigskip

To simplify notation, here we write
$$
\mathbf{E}_k:= \left(\bigoplus_{\alpha \in \mathcal{P}_k(T_D)}\mathcal{L}\left(\mathbf{U}_{\alpha} ,\mathbf{W}_{\alpha}\right) \otimes
\mathrm{span}\{\mathbf{id}_{[\alpha]}\}\right)
$$
for $1 \le k \le \depth(T_D).$ Next, we characterise the elements in the product set
$$
\bigcap_{k=1}^{\depth(T_D)}\xi_{\mathbf{v}}^{(k)}(\mathcal{U}^{(k)}(\mathbf{v})) = \left(\bigcap_{k=1}^{\depth(T_D)}\mathbf{E}_k \right)
 \times \left(\bigcap_{k=1}^{\depth(T_D)} \mathfrak{M}_{\mathbf{r}_k}\left(
 \left. \bigotimes_{\alpha \in \mathcal{P}_{k}(T_D)} \mathbf{U}_{\alpha} \right.
 \right)\right).
$$
Let $\mathcal{O}:=\bigcap_{k=1}^{\depth(T_D)} \mathfrak{M}_{\mathbf{r}_k}\left(
 \left. \bigotimes_{\alpha \in \mathcal{P}_{k}(T_D)} \mathbf{U}_{\alpha} \right.
 \right)$ and $\mathbf{E}:= \bigcap_{k=1}^{\depth(T_D)}\mathbf{E}_k.$ Then we have the following result.
\begin{lemma}
Let $T_D$ be a dimension partition tree with
$\depth(T_D) \ge 2,$ and $\mathfrak{r}=(r_{\alpha})_{\alpha \in T_D} 
\in  \mathcal{AD}(\mathbf{V}_D,T_D)$ such that $\mathcal{FT}_{\mathfrak{r}}(\mathbf{V}_D,T_D)$ is a
proper set of tree-based tensors with a fixed tree-based rank $\mathfrak{r}.$ Assume that
$(\mathbf{V}_{\alpha},\|\cdot\|_{\alpha})$ is a normed space for each $\alpha \in T_D \setminus \{D\}$ 
and that $\|\cdot\|_D$ is a norm 
on the tensor space $\mathbf{V}_D = \left.
\bigotimes_{\alpha \in \mathcal{P}_{k}(T_D)} \mathbf{V}_{\alpha}\right.$ such that \eqref{tree_injective_norm} holds for $1\le k \le \depth(T_D).$ Then for each $\mathbf{v} \in \mathcal{FT}_{\mathfrak{r}}(\mathbf{V}_D,T_D)$
we have that   
$$
\bigcap_{k=1}^{\depth(T_D)}\xi_{\mathbf{v}}^{(k)}(\mathcal{U}^{(k)}(\mathbf{v})) = \mathbf{E} \times \mathcal{O}
$$
is an open set of the Banach space $\mathbf{E} \times \left. \bigotimes_{\delta \in \mathcal{P}_{1}(T_D)} \mathbf{U}_{\delta} \right..$
\end{lemma}
\begin{proof}
First we claim that $\mathcal{O}$ is an open set in $\left. \bigotimes_{\delta \in \mathcal{P}_{1}(T_D)} \mathbf{U}_{\delta} \right..$
To prove the claim, recall that $\mathfrak{M}_{\mathbf{r}_k}\left(
 \left. \bigotimes_{\alpha \in \mathcal{P}_{k}(T_D)} \mathbf{U}_{\alpha} \right.
 \right)$ is an open set in the finite dimensional space $\left. \bigotimes_{\delta \in \mathcal{P}_{k}(T_D)} \mathbf{U}_{\delta} \right.$ for $1 \le k \le \depth(T_D).$ By using Remark~\ref{general_chain} 
we have
$$
\bigcap_{k=1}^{\depth(T_D)} \left. \bigotimes_{\alpha \in \mathcal{P}_{k}(T_D)}\mathbf{U}_{\alpha} \right.
= \bigcap_{k=1}^{\depth(T_D)} \left. \bigotimes_{\alpha \in \mathcal{P}_{k}(T_D)}U_{\alpha}^{\min}(\mathbf{v}) \right.=  \left. \bigotimes_{\delta \in \mathcal{P}_{1}(T_D)} U_{\delta}^{\min}(\mathbf{v}) \right. = \left. \bigotimes_{\delta \in \mathcal{P}_{1}(T_D)} \mathbf{U}_{\delta} \right..
$$
Now, put $\ell = \depth(T_D)$ and consider
$$
\mathcal{O}_{\ell,\ell-1}:= \mathfrak{M}_{\mathfrak{r}_{\ell}}\left( \left. \bigotimes_{\alpha \in \mathcal{P}_{\ell}(T_D)}\mathbf{U}_{\alpha} \right.\right) \cap \mathfrak{M}_{\mathfrak{r}_{\ell-1}}\left( \left. \bigotimes_{\alpha \in \mathcal{P}_{\ell-1}(T_D)}\mathbf{U}_{\alpha} \right.\right)
$$
which is equal to
$$
\mathcal{O}_{\ell,\ell-1} = \left(\mathfrak{M}_{\mathfrak{r}_{\ell}}\left( \left. \bigotimes_{\alpha \in \mathcal{P}_{\ell}(T_D)}\mathbf{U}_{\alpha} \right.\right) \cap \left. \bigotimes_{\alpha \in \mathcal{P}_{\ell -1}(T_D)}\mathbf{U}_{\alpha} \right. \right) \cap \mathfrak{M}_{\mathfrak{r}_{\ell-1}}\left( \left. \bigotimes_{\alpha \in \mathcal{P}_{\ell-1}(T_D)}\mathbf{U}_{\alpha} \right.\right),
$$
where 
$$
\mathfrak{M}_{\mathfrak{r}_{\ell}}\left( \left. \bigotimes_{\alpha \in \mathcal{P}_{\ell}(T_D)}\mathbf{U}_{\alpha} \right.\right) \cap \left. \bigotimes_{\alpha \in \mathcal{P}_{\ell -1}(T_D)}\mathbf{U}_{\alpha} \right.
$$
is an open set in $\left. \bigotimes_{\alpha \in \mathcal{P}_{\ell -1}(T_D)}\mathbf{U}_{\alpha} \right. \subset \left. \bigotimes_{\alpha \in \mathcal{P}_{\ell}(T_D)}\mathbf{U}_{\alpha} \right..$ Next, let
\begin{align*}
\mathcal{O}_{\ell,\ell-2} & = \mathcal{O}_{\ell,\ell-1} \cap \mathfrak{M}_{\mathfrak{r}_{\ell-2}}\left( \left. \bigotimes_{\alpha \in \mathcal{P}_{\ell-2}(T_D)}\mathbf{U}_{\alpha} \right.\right) \\
& = \left(\mathcal{O}_{\ell,\ell-1} \cap \left. \bigotimes_{\alpha \in \mathcal{P}_{\ell-2}(T_D)}\mathbf{U}_{\alpha} \right.\right) \cap \mathfrak{M}_{\mathfrak{r}_{\ell-2}}\left( \left. \bigotimes_{\alpha \in \mathcal{P}_{\ell-2}(T_D)}\mathbf{U}_{\alpha} \right.\right).
\end{align*}
In a similar way as above, $\mathcal{O}_{\ell,\ell-2}$ is an open set in $\left. \bigotimes_{\alpha \in \mathcal{P}_{\ell-2}(T_D)}\mathbf{U}_{\alpha} \right..$ By induction, we prove that
$\mathcal{O}= \mathcal{O}_{\ell,1}$ is an open set in $\left. \bigotimes_{\delta \in \mathcal{P}_{1}(T_D)} U_{\delta}^{\min}(\mathbf{v}) \right.$ and the claim follows. To conclude, from Lemma~\ref{G1}, $\mathbf{E}_k$ is a closed linear space of $\mathcal{L}(\mathbf{V}_{D_{\|\cdot\|_D}},\mathbf{V}_{D_{\|\cdot\|_D}})$ for $1 \le k \le \depth(T_D).$ Hence $\mathbf{E}:= \bigcap_{k=1}^{\depth(T_D)}\mathbf{E}_k$ is a linear closed subspace in the Banach space $\mathcal{L}(\mathbf{V}_{D_{\|\cdot\|_D}},\mathbf{V}_{D_{\|\cdot\|_D}}).$ Thus, $\mathbf{E}$ is also a Banach space. 
Since $\mathbf{E} \times \mathcal{O}$ is an open set in the Banach space $\mathbf{E} \times \left. \bigotimes_{\delta \in \mathcal{P}_{1}(T_D)} \mathbf{U}_{\delta} \right.$ the lemma follows.
\end{proof}

\bigskip

Given $L \in \mathbf{E}$, for each $1 \le k \le \depth(T)$ there exists a unique 
$$
(L_{\alpha}^{(k)})_{\alpha \in \mathcal{P}_{k}(T_D)} \in 
\mathop{\mathchoice{\raise-0.22em\hbox{\huge
 $\times$}} {\raise-0.05em\hbox{\Large $\times$}}{\hbox{\large
 $\times$}}{\times}}_{\alpha\in \mathcal{P}_{k}(T_D)}\mathcal{L}(\mathbf{U}_{\alpha},\mathbf{W}_{\alpha})
$$
such that
$$
L = \Delta((L_{\alpha}^{(k)})_{\alpha \in \mathcal{P}_{k}(T_D)}) = \sum_{\alpha \in \mathcal{P}_{k}(T_D)} L_{\alpha}^{(k)} \otimes \mathbf{id}_{[\alpha]}
$$
holds. From \eqref{igualdad_fundamental}, each $(L,\mathbf{u}) \in \mathbf{E} \times \mathcal{O}$ satisfies that
$$
(\xi_{\mathbf{v}}^{(k)})^{-1}((L,\mathbf{u})) = \exp(L)(\mathbf{u}) \in \mathcal{U}^{(k)}(\mathbf{v})
$$
for $1 \le k \le \depth(T_D).$ Hence the image of $(L,\mathbf{u})$ by
$(\xi_{\mathbf{v}}^{(k)})^{-1}$ is independent on
the index $k.$ Thus
$(\xi_{\mathbf{v}}^{(k)})^{-1}$ is a bijection that maps
$\mathbf{E} \times \mathcal{O}$ onto a subset
$\mathcal{W}(\mathbf{v}) \subset \bigcap_{l=1}^{\depth(T_D)}\mathcal{U}^{(l)}(\mathbf{v})$ containing $\mathbf{v}$ for each $1 \le k \le \depth(T_D).$ It allows to we define the bijection
$$
\boldsymbol{\xi}_{\mathbf{v}}:\mathcal{W}(\mathbf{v}) \longrightarrow \mathbf{E} \times \mathcal{O}
$$
by $\boldsymbol{\xi}_{\mathbf{v}}(\mathbf{w}) = \xi_{\mathbf{v}}^{(k)}(\exp(L)(\mathbf{u})) = (L,\mathbf{u}).$

\bigskip

Then the following result is straightforward.

\begin{theorem}
\label{TBF_Banach_Manifold} 
Let $T_D$ be a dimension partition tree with
$\depth(T_D) \ge 2,$ and $\mathfrak{r}=(r_{\alpha})_{\alpha \in T_D} 
\in  \mathcal{AD}(\mathbf{V}_D,T_D)$ such that $\mathcal{FT}_{\mathfrak{r}}(\mathbf{V}_D,T_D)$ is a
proper set of tree-based tensors with a fixed tree-based rank $\mathfrak{r}.$ Assume that
$(\mathbf{V}_{\alpha},\|\cdot\|_{\alpha})$ is a normed space for each $\alpha \in T_D \setminus \{D\}$ 
and that $\|\cdot\|_D$ is a norm 
on the tensor space $\mathbf{V}_D = \left.
\bigotimes_{\alpha \in \mathcal{P}_{k}(T_D)} \mathbf{V}_{\alpha}\right.$ is such that \eqref{tree_injective_norm} holds for $1\le k \le \depth(T_D).$ Then the collection
$$\mathcal{B} = \{(\mathcal{W}(\mathbf{v}),\boldsymbol{\xi}_{\mathbf{v}})\}_{\mathbf{v} \in \mathcal{FT}_{\mathfrak{r}}(\mathbf{V}_{D},T_D)}$$ is a  $\mathcal{C}^{\infty}$-atlas for $\mathcal{FT}_{\mathfrak{r}}(%
\mathbf{V}_{D},T_D)$,
and hence  $\mathcal{FT}_{\mathfrak{r}}(%
\mathbf{V}_{D},T_D)$  is a $\mathcal{C}^{\infty}$-Banach manifold modelled on
$$
\mathbf{E} \times \mathbb{R}^{
\bigtimes_{\alpha \in \mathcal{P}_{1}(T_D)} r_{\alpha}}.
$$
Here
$ \mathbf{U}_{\alpha}$ is a $r_{\alpha}$-dimensional subspace of $\mathbf{V}_{\alpha}$ for each $\alpha \in T_D \setminus \{D\}$ where $\mathbf{v} \in  \left. \bigotimes_{\alpha \in \mathcal{P}_k(T_D)}\mathbf{U}_{\alpha} \right.$ for $1 \le k \le \depth(T_D)$ and $\mathbf{W}_\alpha$ is a closed subspace of $\mathbf{V}_{\alpha_{\|\cdot\|_{\alpha}}}$ such that $\mathbf{V}_{\alpha_{\|\cdot\|_{\alpha}}} = \mathbf{U}_{\alpha} \oplus \mathbf{W}_{\alpha} ,$ where $\mathbf{V}_{{\alpha}_{\|\cdot\|_{\alpha}}}$ is the completion of
$\mathbf{V}_{\alpha}$ for $\alpha \in T_D \setminus \{D\}.$
\end{theorem}

\subsection{$\mathcal{FT}_{\mathfrak{r}}(%
\mathbf{V}_{D},T_D)$ as embedded sub-manifold of $\mathfrak{M}_{\mathfrak{r}_{{k}}}(%
\mathbf{V}_{D},\mathcal{P}_k(T_D))$ for $1 \le k \le \depth(T_D)$ }

Since $\mathcal{FT}_{\mathfrak{r}}(%
\mathbf{V}_{D},T_D) \subset \mathfrak{M}_{\mathfrak{r}_{{k}}}(%
\mathbf{V}_{D},\mathcal{P}_k(T_D)),$ for $1 \le k \le \depth(T_D),$ the natural ambient space 
of the manifold $\mathcal{FT}_{\mathfrak{r}}(%
\mathbf{V}_{D},T_D)$ is any manifold $\mathfrak{M}_{\mathfrak{r}_{{k}}}(%
\mathbf{V}_{D},\mathcal{P}_k(T_D))$ for $1 \le k \le \depth(T_D).$ 
In order to prove that $\mathcal{FT}_{\mathfrak{r}}(%
\mathbf{V}_{D},T_D)$ is an embedded sub-manifold of $\mathfrak{M}_{\mathfrak{r}_{{k}}}(%
\mathbf{V}_{D},\mathcal{P}_k(T_D))$ for $1 \le k \le \depth(T_D)$, we consider
the natural inclusion map $\mathfrak{i}: \mathcal{FT}_{\mathfrak{r}}(%
\mathbf{V}_{D},T_D)\longrightarrow \mathfrak{M}_{\mathfrak{r}_{{k}}}(%
\mathbf{V}_{D},\mathcal{P}_k(T_D))$ given by $\mathfrak{i}(\mathbf{v}) = \mathbf{v}.$ 
Then, from Theorem 3.5.7 in \cite{Lang}, we only 
need to check the following two conditions for each $1 \le k \le \depth(T_D):$ 
\begin{enumerate}
    \item[(C1)] The map $\mathfrak{i}$ should be an immersion. From Proposition 4.1 in \cite{FHN}, it is true when
    the linear map $$\mathrm{T}_{\mathbf{v}}\mathfrak{i} = (\xi_{\mathbf{v}} \circ \mathfrak{i} \circ \boldsymbol{\xi}_{\mathbf{v}}^{-1})'(\boldsymbol{\xi}_{\mathbf{v}}(\mathbf{v})):\mathrm{T}_{\mathbf{v}}\mathcal{FT}_{\mathfrak{r}}(%
    \mathbf{V}_{D},T_D) \longrightarrow \mathrm{T}_{\mathbf{v}} \mathfrak{M}_{\mathfrak{%
    r}_{k}}(\mathbf{V}_{D},\mathcal{P}_k(T_D))$$ is injective and
$\mathrm{T}_{\mathbf{v}}\mathfrak{i}(\mathrm{T}_{\mathbf{v}}\mathcal{FT}_{\mathfrak{r}}(%
\mathbf{V}_{D},T_D)) \in \mathbb{G}\left( \mathrm{T}_{\mathbf{v}} \mathfrak{M}_{\mathfrak{%
r}_{k}}(\mathbf{V}_{D},\mathcal{P}_k(T_D))\right)$ 
\item[(C2)]The map
$$
\mathfrak{i}: \mathcal{FT}_{\mathfrak{r}}(%
\mathbf{V}_{D},T_D) \longrightarrow \mathfrak{i}\left( \mathcal{FT}_{\mathfrak{r}}(%
\mathbf{V}_{D},T_D) \right)
$$
is a topological homeomorphism.
\end{enumerate}

Since $\mathfrak{i}:\mathcal{FT}_{\mathfrak{r}}(%
\mathbf{V}_{D},T_D) \longrightarrow \mathfrak{i}\left( \mathcal{FT}_{\mathfrak{r}}(%
\mathbf{V}_{D},T_D) \right)$ is the identity map then it is clearly an homeomorphism and (C2) holds. To prove that (C1) is also true, first we claim that the natural inclusion map $\mathfrak{i}$ is also written 
in local coordinates as the natural inclusion map.
Indeed, for $\mathbf{v} \in \mathcal{FT}_{\mathfrak{r}}(%
\mathbf{V}_{D},T_D),$ the open set $\mathcal{W}(\mathbf{v}) \subset \bigcap_{\ell=1}^{\depth(T_D)} \mathcal{U}^{(\ell)}(\mathbf{v}) \subset \mathcal{FT}_{\mathfrak{r}}(%
\mathbf{V}_{D},T_D)$ and hence
$$
\mathfrak{i}:\mathcal{W}(\mathbf{v}) \longrightarrow
\mathcal{U}^{(k)}(\mathbf{v})
$$
is the identity map on $\mathcal{W}(\mathbf{v})$, that is, $\mathfrak{i}|_{\mathcal{W}(\mathbf{v})}= id_{\mathcal{W}(\mathbf{v})}.$ Thus
$$
(\xi_{\mathbf{v}}^{(k)} \circ \mathfrak{i} \circ \boldsymbol{\xi}^{-1}_{\mathbf{v}}):
\mathbf{E} \times \mathcal{O} \longrightarrow
\mathbf{E}_k \times \mathfrak{M}_{\mathfrak{r}_{k}}\left(
\left. \bigotimes_{\alpha \in \mathcal{P}_{k}(T_D)} \mathbf{U}_{\alpha} \right.
\right)
$$
is the natural inclusion map and the claim follows. Hence its derivative
$$
\mathrm{T}_{\mathbf{v}}\mathfrak{i} = (\xi_{\mathbf{v}} \circ \mathfrak{i} \circ \boldsymbol{\xi}_{\mathbf{v}}^{-1})'(\boldsymbol{\xi}_{\mathbf{v}}(\mathbf{v})):
\mathbf{E} \times \left(\left. \bigotimes_{\alpha \in \mathcal{P}_{1}(T_D)}\mathbf{U}_{\alpha} \right.\right) \longrightarrow
\mathbf{E}_k \times \left(
\left. \bigotimes_{\alpha \in \mathcal{P}_{k}(T_D)} \mathbf{U}_{\alpha} \right.
\right)
$$ 
is also the natural inclusion map which is clearly injective.

\bigskip

In consequence, to obtain (C1) we only need to prove that for each
$\mathbf{v} \in \mathcal{FT}_{\mathfrak{r}}(%
\mathbf{V}_{D},T_D)$ the tangent space
$$
\mathrm{T}_{\mathbf{v}}\mathcal{FT}_{\mathfrak{r}}(%
\mathbf{V}_{D},T_D) = \mathbf{E} \times \left. \bigotimes_{\alpha \in \mathcal{P}_{1}(T_D)} \mathbf{U}_{\alpha} \right.
$$
belongs to
$$
\mathbb{G}\left(\mathbf{E}_k
\times 
 \left. \bigotimes_{\alpha \in \mathcal{P}_{k}(T_D)} \mathbf{U}_{\alpha} \right.\right)
 = \mathbb{G}\left( \mathbf{E}_k\right)
 \times \mathbb{G}\left(\left. \bigotimes_{\alpha \in \mathcal{P}_{k}(T_D)} \mathbf{U}_{\alpha} \right. \right).
$$
Clearly
$$
\left. \bigotimes_{\alpha \in \mathcal{P}_{1}(T_D)} \mathbf{U}_{\alpha} \right. 
\in \mathbb{G}\left(\left. \bigotimes_{\alpha \in \mathcal{P}_{k}(T_D)} \mathbf{U}_{\alpha} \right. \right),
$$
because $\left. \bigotimes_{\alpha \in \mathcal{P}_{k}(T_D)} \mathbf{U}_{\alpha} \right. $ is a finite dimensional vector space. From Lemma~\ref{G1} we have $$\mathbf{E}_k \in \mathbb{G}\left( \mathcal{L}(\mathbf{V}_{D_{\|\cdot\|_D}},\mathbf{V}_{D_{\|\cdot\|_D}}) \right)$$ for $1 \le k \le \depth(T_D).$  The second statement of Lemma 4.14 in \cite{FHN} implies
$$
\mathbf{E} = \bigcap_{k=1}^{\depth(T_D)}\mathbf{E}_k \in 
\mathbb{G}\left(\mathcal{L}(\mathbf{V}_{D_{\|\cdot\|_D}},\mathbf{V}_{D_{\|\cdot\|_D}}) \right).
$$
Thus we have the following theorem.

\begin{theorem}\label{submanifold}
Let $T_{D}$ be a dimension partition tree over $D$ and $\mathfrak{r} \in \mathcal{AD}(\mathbf{V}_D,T_D)$ such that $\mathcal{FT}_{\mathfrak{r}}(\mathbf{V}_D,T_D)$ is a
proper set of tree-based tensors with a fixed tree-based rank $\mathfrak{r}.$ Assume that $(\mathbf{V}_{\alpha},\|\cdot\|_{\alpha})$ is a normed space 
for each $\alpha \in T_D \setminus \{D\}$ 
and let $\|\cdot\|_D$ be a norm 
on the tensor space $\mathbf{V}_D$ such that \eqref{tree_injective_norm}
holds for $1 \le k \le \depth(T_D).$ Then $\mathcal{FT}_{\mathfrak{r}}(%
\mathbf{V}_{D},T_D)$ is an embedded sub-manifold of $\mathfrak{M}_{\mathfrak{r}_{{k}}}(%
\mathbf{V}_{D},\mathcal{P}_k(T_D))$ for $1 \le k \le \depth(T_D).$
\end{theorem}

\bigskip

Observe that we can also consider the natural inclusion map $\mathfrak{i}$ from $\mathfrak{M}_{\mathfrak{r}_{{k}}}(%
\mathbf{V}_{D},\mathcal{P}_k(T_D))$ to $\mathbf{V}_{D_{\|\cdot\|_D}}.$  Under the assumptions of Theorem~\ref{submanifold},
by using Theorem 4.14 of \cite{FHN}, we have that $\mathfrak{M}_{\mathfrak{r}_{{k}}}(%
\mathbf{V}_{D},\mathcal{P}_k(T_D))$ is an immersed sub-manifold of $\mathbf{V}_{D_{\|\cdot\|_D}}$ and, for each $\mathbf{v} \in \mathfrak{M}_{\mathfrak{r}_{{k}}}(%
\mathbf{V}_{D},\mathcal{P}_k(T_D)),$ the tangent
space $$\mathrm{T}_{\mathbf{v}}\mathfrak{M}_{\mathfrak{r}_{{k}}}(%
\mathbf{V}_{D},\mathcal{P}_k(T_D)) = \mathbf{E}_k \times \left(\left. \bigotimes_{\alpha \in  \mathcal{P}_k(T_D)}\mathbf{U}_{\alpha} \right. \right)$$ is linearly isomorphic to the linear space
$
\mathrm{T}_{\mathbf{v}}\mathfrak{i}\left( 
\mathrm{T}_{\mathbf{v}}\mathfrak{M}_{\mathfrak{r}_{{k}}}(%
\mathbf{V}_{D},\mathcal{P}_k(T_D))
\right) \in \mathbb{G}(\mathbf{V}_{D_{\|\cdot\|_D}}).
$
Moreover,
\begin{align*}
\bigcap_{k=1}^{\depth(T_D)}\mathrm{T}_{\mathbf{v}}\mathfrak{M}_{\mathfrak{r}_{{k}}}(%
\mathbf{V}_{D},\mathcal{P}_k(T_D)) & = \bigcap_{k=1}^{\depth(T_D)}\left( \mathbf{E}_k \times  \left(\left. \bigotimes_{\alpha \in \mathcal{P}_{k}(T_D)} \mathbf{U}_{\alpha} \right.\right)\right) \\ 
& = \left( \bigcap_{k=1}^{\depth(T_D)}\mathbf{E}_k \right) \times  \left(\bigcap_{k=1}^{\depth(T_D)} \left. \bigotimes_{\alpha \in \mathcal{P}_{k}(T_D)} \mathbf{U}_{\alpha} \right.\right) \\ 
& = \mathbf{E} \times \left(\left. \bigotimes_{\alpha \in \mathcal{P}_{1}(T_D)} \mathbf{U}_{\alpha} \right.\right) = \mathrm{T}_{\mathbf{v}}\mathcal{FT}_{\mathfrak{r}}(%
\mathbf{V}_{D},T_D).
\end{align*}
Then, by using that $\mathrm{T}_{\mathbf{v}}\mathfrak{i}$ is injective, we obtain 
\begin{align*}
\mathrm{T}_{\mathbf{v}}\mathfrak{i} \left(
\mathrm{T}_{\mathbf{v}}\mathcal{FT}_{\mathfrak{r}}(%
\mathbf{V}_{D},T_D) \right) & = \mathrm{T}_{\mathbf{v}}\mathfrak{i} \left(\bigcap_{k=1}^{\depth(T_D)}\mathrm{T}_{\mathbf{v}}\mathfrak{M}_{\mathfrak{r}_{{k}}}(%
\mathbf{V}_{D},\mathcal{P}_k(T_D))\right) \\ 
& = \bigcap_{k=1}^{\depth(T_D)} \mathrm{T}_{\mathbf{v}}\mathfrak{i} \left(\mathrm{T}_{\mathbf{v}}\mathfrak{M}_{\mathfrak{r}_{{k}}}(%
\mathbf{V}_{D},\mathcal{P}_k(T_D))\right)  \in \mathbb{G}(\mathbf{V}_{D_{\|\cdot\|_D}}),
\end{align*}
also by Lemma 4.14 in \cite{FHN}, and it is linearly isomorphic 
to $\mathrm{T}_{\mathbf{v}}\mathcal{FT}_{\mathfrak{r}}(%
\mathbf{V}_{D},T_D).$  Thus, also $\mathcal{FT}_{\mathfrak{r}}(%
\mathbf{V}_{D},T_D)$ is an immersed sub-manifold of $\mathbf{V}_{D_{\|\cdot\|_D}}.$ Hence we have the following
result.

\begin{corollary}\label{immersed1}
Let $T_{D}$ be a dimension partition tree over $D$ and $\mathfrak{r} \in \mathcal{AD}(\mathbf{V}_D,T_D)$ such that $\mathcal{FT}_{\mathfrak{r}}(\mathbf{V}_D,T_D)$ is a
proper set of tree-based tensors with a fixed tree-based rank $\mathfrak{r}.$ Assume that $(\mathbf{V}_{\alpha},\|\cdot\|_{\alpha})$ is a normed space 
for each $\alpha \in T_D \setminus \{D\}$ 
and let $\|\cdot\|_D$ be a norm 
on the tensor space $\mathbf{V}_D$ such that \eqref{tree_injective_norm}
holds for $1 \le k \le \depth(T_D).$ Then $\mathcal{FT}_{\mathfrak{r}}(%
\mathbf{V}_{D},T_D)$ is an immersed sub-manifold of $\mathbf{V}_{D_{\|\cdot\|_D}}.$ 
\end{corollary}

\subsection{On the Dirac–Frenkel Variational Principle}

To extend Dirac–Frenkel Variational Principle for a proper set of tree-based tensors with a fixed tree-based rank,  
we consider the abstract ordinary differential equation in a
reflexive tensor Banach space $\mathbf{V}_{D_{\Vert \cdot \Vert _{D}}} = \overline{\mathbf{V}_D}^{\|\cdot\|_D},$
given by 
\begin{align}
\dot{\mathbf{u}}(t)& =\mathbf{F}(t,\mathbf{u}(t)),\text{ for }t\geq 0,
\label{BODE1} \\
\mathbf{u}(0)& =\mathbf{u}_{0},  \label{BODE2}
\end{align}%
where we assume $\mathbf{u}_{0}\neq \mathbf{0}$ and $\mathbf{F}:[0,\infty
)\times \mathbf{V}_{D_{\Vert \cdot \Vert _{D}}}\longrightarrow \mathbf{V}%
_{D_{\Vert \cdot \Vert _{D}}}$ satisfying the usual conditions 
to have existence and uniqueness of solutions. Let $T_{D}$ be a dimension partition tree over $D$ and $\mathfrak{r} \in \mathcal{AD}(\mathbf{V}_D,T_D)$ such that $\mathcal{FT}_{\mathfrak{r}}(\mathbf{V}_D,T_D)$ is a
proper set of tree-based tensors with a fixed tree-based rank $\mathfrak{r}.$ Assume that $(\mathbf{V}_{\alpha},\|\cdot\|_{\alpha})$ is a normed space 
for each $\alpha \in T_D \setminus \{D\}$ 
and let $\|\cdot\|_D$ be a norm 
on the tensor space $\mathbf{V}_D$ such that \eqref{tree_injective_norm}
holds for $1 \le k \le \depth(T_D).$ 

\bigskip

We want to approximate $%
\mathbf{u}(t),$ for $t\in I:=(0,T )$ for some $T >0,$ by
a differentiable curve $t\mapsto \mathbf{v}_{r}(t)$ from $I$ to
$\mathcal{FT}_{\mathfrak{r}}(\mathbf{V}_D,T_D),$ where $\mathfrak{r}\in \mathcal{AD}(\mathbf{V}_D,T_D)$
$(\mathfrak{r} \neq \mathbf{0}),$ such that $\mathbf{v}_{r}(0)=\mathbf{v}_{0}\in \mathcal{FT}_{\mathfrak{r}}(\mathbf{V}_D,T_D)$ is an approximation of $\mathbf{u}_{0}.$

\bigskip

To construct a reduced order model of \eqref{BODE1}--\eqref{BODE2} over 
the Banach manifold $\mathcal{FT}_{\mathfrak{r}}(\mathbf{V}_D,T_D)$ we consider
the natural inclusion map
$$
\mathfrak{i}:\mathcal{FT}_{\mathfrak{r}}(%
\mathbf{V}_{D},T_D) \longrightarrow \mathbf{V}_{D_{\Vert \cdot \Vert _{D}}}.
$$ 
Since $\mathcal{FT}_{\mathfrak{r}}(%
\mathbf{V}_{D},T_D)$ is an immersed sub-manifold in $\mathbf{V}_{D_{\|\cdot\|_D}},$  
from Theorem 3.5.7 in \cite{Lang}, we have
$$
\mathrm{T}_{\mathbf{v}}\mathfrak{i} \left(
\mathrm{T}_{\mathbf{v}}\mathcal{FT}_{\mathfrak{r}}(%
\mathbf{V}_{D},T_D) \right) \in \mathbb{G}(\mathbf{V}_{D_{\|\cdot\|_D}}).
$$
By using that $\mathbf{F}(t,\mathbf{v}_{r}(t))\in \mathbf{V}_{D_{\Vert
\cdot \Vert _{D}}},$ for each $t\in I,$ together the fact that
$$\mathbf{Z}^{(D)}(\mathbf{v}_{r}(t)):= \mathrm{T}_{\mathbf{v}_{r}(t)}\mathfrak{i} \left(
\mathrm{T}_{\mathbf{v}_{r}(t)}\mathcal{FT}_{\mathfrak{r}}(\mathbf{V}_{D},T_D) \right)$$ 
is a closed linear subspace in $\mathbf{V}_{D_{\Vert \cdot \Vert
_{D}}},$ we have the existence of a $\dot{\mathbf{v}}_{r}(t)\in \mathbf{Z}%
^{(D)}(\mathbf{v}_{r}(t))$ such that 
\begin{equation}\label{DF}
\Vert \dot{\mathbf{v}}_{r}(t)-\mathbf{F}(t,\mathbf{v}_{r}(t))\Vert
_{D}=\min_{\dot{\mathbf{v}}\in \mathbf{Z}^{(D)}(\mathbf{v}_{r}(t))}\Vert 
\dot{\mathbf{v}}-\mathbf{F}(t,\mathbf{v}_{r}(t))\Vert _{D}.
\end{equation}
Equation \eqref{DF} extends the variational principle of Dirac-Frenkel
to the Banach manifold $\mathcal{FT}_{\mathfrak{r}}(\mathbf{V}_{D},T_D).$

\bigskip

\textbf{Acknowledgements} This research was funded by the RTI2018-093521-B-C32 grant from the Ministerio de Ciencia, Innovación y Universidades and by the grant number INDI22/15 from Universidad CEU Cardenal Herrera.

\end{document}